\documentclass[amsfonts]{article}

\usepackage{graphics,amsmath,amsfonts,amssymb,graphicx,url,color,amsthm,cite,array,mathtools,booktabs}

\usepackage[margin=10pt,font=small,labelfont=bf,
labelsep=period]{caption}

\usepackage{hyperref}

\newlength{\notewidth}
\usepackage{varwidth}

\usepackage[shortlabels ]{enumitem}
\usepackage[all]{xy}
\usepackage[nottoc]{tocbibind}

\newtheorem{lemma}{Lemma}[section]
\newtheorem{proposition}[lemma]{Proposition}
\newtheorem{theorem}{Theorem}
\newtheorem{definition}[lemma]{Definition}

\theoremstyle{definition}
\newtheorem{remark}[lemma]{Remark}

\newcommand{\e}{{\mathbf e}}
\renewcommand{\r}{{\mathbf{r}}}

\newcommand{\z}{{\mathbf{z}}}

\newcommand{\hb}{\mathbf{h} }

\newcommand{\diag}{\mathrm{diag}}
\newcommand{\SL}{\mathrm{SL}}

\newcommand{\GL}{\mathrm{GL}}

\newcommand{\SLt}{{\SL_2(\R)}}

\newcommand{\End}{{\rm End}}
\newcommand{\Ker}{{\rm Ker}}

\newcommand{\SUt}{{\mathrm{SU}_2 }}
\newcommand{\SE}{\mathrm{SE}_2 }
\newcommand{\Span}{\mathrm{Span}}
\newcommand{\id}{\mathrm{id}}

\newcommand{\slt}{\mathfrak{sl}_2(\R)}

\newcommand{\g}{\mathfrak{g}}

\newcommand{\su}{\mathfrak{su}}
\newcommand{\h}{{\mathfrak h}}
\newcommand{\se}{\mathfrak{se}_2 }

\newcommand{\cE}{\mathcal E}

\newcommand{\R}{\mathbb{R}}
\newcommand{\C}{\mathbb{C}}
\renewcommand{\H}{\mathbb{H}}
\newcommand{\Z}{\mathbb{Z}}

\newcommand{\RP}{\R P}

\newcommand{\st}{\, | \,}

\newcommand{\n}{\noindent}
\newcommand{\bs}{\bigskip}

\newcommand{\mn}{\medskip\noindent}
\newcommand{\bn}{\bs\n}

\newcommand{\be}{\begin{equation}}
\newcommand{\ee}{\end{equation}}

\renewcommand{\d}{\mathrm{d}}
\newcommand{\PSL}{\mathrm{ PSL}}
\newcommand{\SU}{\mathrm{SU}}

\renewcommand{\>}{\rangle}
\newcommand{\<}{\langle}
\newcommand{\etab}{\eta}

 \newcommand{\benum}{\begin{enumerate}[label=$(\mathrm{\alph*})$, left=-5px, itemsep=-2px]}

\newcounter{ticount}

\newcommand{\s}{\small}

\newcommand{\M}{P}
\newcommand{\ad}{\mathrm{ad}}
\newcommand{\grm}{\mathrm{g}}
\usepackage{tikz-cd,import,pdfpages,eqnarray,}

\renewcommand{\S}{\mathbb S}

\renewcommand{\P}{\mathbb P}

\title {Chains of path geometries on surfaces:\\ theory and examples}

\author{Gil Bor\footnote{
CIMAT, A.P. 402, Guanajuato, Gto. 36000, Mexico;
gil@cimat.mx
}
\and
Travis Willse\footnote{
Guilford College, 5800 W Friendly Ave, Greensboro, NC 27410, USA;
twillse@guilford.edu
}
}

\date{\today}
\begin{document}
\maketitle

\begin{abstract}
We derive the equations of chains for path geometries on surfaces by solving  the equivalence problem of a related structure: sub-Riemannian geometry of signature $(1,1)$ on a contact 3-manifold. This approach is significantly simpler than the standard method of solving the full equivalence problem for path geometry. We then use these equations to give a characterization of projective path geometries in terms of their chains (the chains projected to the surface coincide with the paths) and study the chains of four examples of homogeneous
 path geometries. In one of these examples (horocycles in the hyperbolic planes) the projected chains are bicircular quartics.

\end{abstract}
\tableofcontents
\section{Introduction}
\subsection{A quick reminder about path geometries on surfaces}

A {\em path geometry} on a surface  consists of a surface $\Sigma$ (a  2-dimensional differentiable manifold)  together with a non-degenerate 2-parameter family of {\em unparametrized} curves in $\Sigma.$\footnote{This definition  will be reformulated below more abstractly and precisely; in particular, the non-degeneracy condition will be spelled out.}An {\em equivalence} of path geometries on  two surfaces  is a diffeomorphism of the surfaces which maps the paths of one surface  onto those of the other. A {\em symmetry} of a path geometry on a surface is a self-equivalence.

The basic example is $\Sigma=\RP^2$ (the $2$-dimensional real projective plane) equipped with the family of straight lines in it. A path geometry\footnote{We shall henceforth usually drop the qualifier ``on a surface'' since that is the only situation this article considers.} which is locally equivalent to this example is called {\em flat}. A less obvious flat example is given by all parabolas whose focus is at the origin (`Kepler parabolas'; here   $\Sigma:=\R^2\setminus \{0\}$). It is doubly covered by straight lines via the (complex) quadratic map $z\mapsto z^2$.

\begin{figure}
    \centering
    \includegraphics[width=\textwidth]{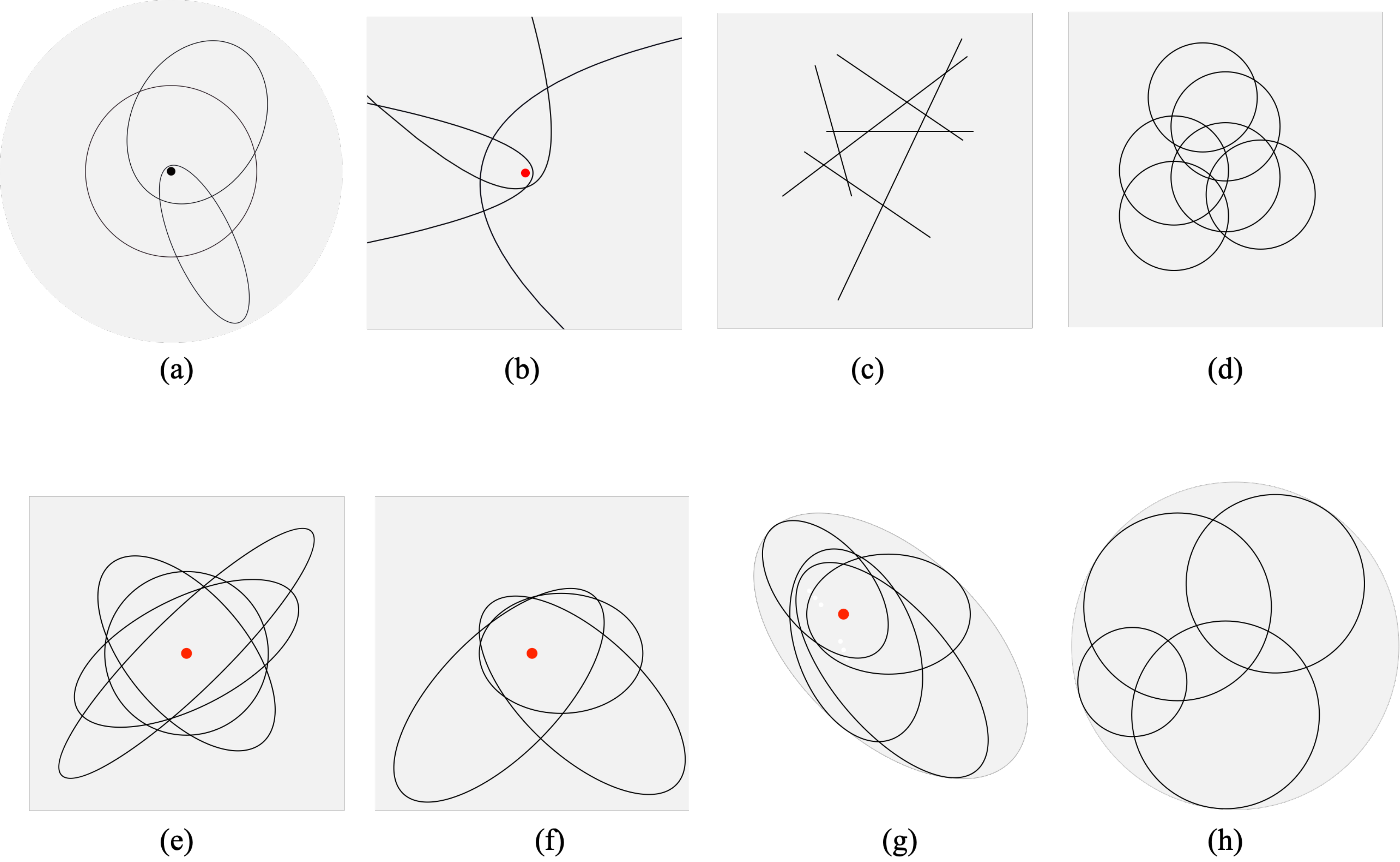}
    \caption{
    A gallery of 2D path geometries:
    (a)~Kepler ellipses of fixed major axis.
    (b)~Kepler parabolas.
    (c)~Straight lines.
    (d)~Circles of fixed radius.
    (e)~Hooke ellipses of fixed area.
    (f)~Kepler ellipses of fixed minor axis.
    (g)~Kepler ellipses tangent to a fixed Kepler ellipse.
    (h)~Circles tangent to a fixed circle (horocycles).
    Can you find the equivalent and dual geometries? \\
    \footnotesize(Answer: $ a=e=f, b=c=g$ (flat); $a^*=h, b^*=b, d^*=d.$)}

    \label{fig:gallery}
\end{figure}

Every path geometry is given locally by the graphs of solutions of a second-order ODE $y''=f(x,y,y')$. Conversely, a   path geometry determines the ODE up to so-called point transformations, that is, changes of coordinate $(x, y) \mapsto (\tilde x, \tilde y)$. The flat example of straight lines in $\RP^2$ corresponds to $y''=0$.     A path geometry is {\em projective} if its paths are the (unparametrized) geodesics of a torsion-free affine  connection on $\Sigma$. Such path geometries  correspond to  ODEs $y''=f(x,y,y')$ where $f$   is at most  cubic in $y'$. Note that, somewhat  surprisingly, this condition is independent of the coordinates  $x,y$  used on $\Sigma.$  Thus a `generic' path geometry is not projective, and in particular,  non-flat.  A non-projective example is  the  path geometry in $\R^2$ whose paths are all circles of a fixed radius.

A path geometry on a surface $\Sigma$ defines a {\em dual} path geometry on the  path space  $\Sigma^*$, whose paths are parametrized by $\Sigma$: for each point  $x\in\Sigma$ the corresponding   path in $\Sigma^*$ consists of all paths in  $\Sigma$ passing through $x$. Clearly, the dual of a flat path geometry is flat as well,   an example of a {\em self-dual} path geometry. The path geometry of circles of  fixed radius in $\R^2$ is an example of a self-dual non-projective path geometry. A projective path geometry is flat if and only if  its  dual is projective as well.

A flat path geometry admits an 8-dimensional (local) group of symmetries (the projective group $\PSL_3(\R)$). Conversely, a path geometry admitting an 8-dimensional local group of symmetries is necessarily flat (a theorem of Sophus Lie).  The {\em sub-maximal}  symmetry dimension, i.e. the maximum dimension of the local symmetry group of a non-flat path geometry,  is 3.  The path geometry of  circles with a fixed radius  is  sub-maximal.   Its  symmetry  group  is the Euclidean group. Another sub-maximal example  is given by central ellipses  (`Hooke ellipses') of fixed area. The  symmetry group $\SLt$, acting by its standard linear action on $\R^2$ (here $\Sigma=\R^2\setminus \{0\}$). In contrast to the previous example of circles with fixed radius, this example is projective and  non--self-dual: its dual  is the hyperbolic plane with the set of  horocycles   (in the Poincar\'e disk or upper half-plane model horocycles are precisely the circles tangent to the boundary).  The set of Kepler conics of fixed major or minor axis  (either hyperbolas or ellipses, with one of their foci at the origin) defines an interesting  path geometry which is locally equivalent to that of Kepler ellipses of fixed area,  see \cite{BJ1}.

  The subject was studied extensively  in the second half of the 19th century by Roger Liouville (a relative of the more famous Joseph Liouville), Sophus Lie and his student Arthur Tresse, who produced a local classification, over the complex numbers, of sub-maximal path geometries (i.e. those admitting a 3-dimensional group of symmetries) \cite{T}. This classification has been since refined over the real numbers \cite{D}. The only non-flat projective items on the list is the above mentioned  case of central ellipses of fixed area  (equivalently, Kepler ellipses  of fixed major or minor axis) and  central hyperbolas of fixed discriminant  (equivalently, Kepler hyperbolas of fixed minor axis; see Table 2 in the Appendix of  \cite{BJ}).

\subsection{An abstract reformulation of path geometry}\label{sect:abs_reform}
 We   describe here briefly    a more
  abstract and rigorous reformulation of  path geometries on surfaces, useful  also for  introducing chains. For further details  we recommend V. I. Arnol'd's book \cite[Chapter 1, Section 6]{arnold}.

Given a surface $\Sigma$, let  $\P T\Sigma$ be the ($3$-dimensional) total space of its projectivized tangent bundle. That is, a point in $\P T\Sigma$ corresponds to a point in $\Sigma$ together with a tangent line at the point
      (a 1-dimensional  linear subspace of the tangent space at the point). There is a standard contact distribution $D$ on $\P T\Sigma$, given by the  `skating' condition: ``the point moves along the line'', or ``the line rotates about the point.'' The fibers of the base point projection $\P T\Sigma\to \Sigma$ are integral curves of $D$. Their tangents form the vertical line field $L_1\subset D$.
       A path $ \gamma\subset \Sigma$  is lifted to $\P T\Sigma$ by mapping
        a  point on $\gamma$ to  the  tangent line to $\gamma$ at this point. The lifted curve is clearly an integral curve of $D$, as it satisfies the skating condition.
         The non-degeneracy assumption on a  path geometry on $\Sigma$
          is   that the {\em lifted paths form a smooth 1-dimensional foliation of   $\P T\Sigma$},  transverse to $L_1$ (in $D$); equivalently,  the tangent lines to the lifted curves form  a smooth
           line field $L_2\subset D$, complementary to $L_1$, so that  $D=L_1\oplus L_2$.

We thus arrive at an abstract  reformulation of a path geometry:

\begin{definition}\label{def:pg} A (2-dimensional) path geometry is  a smooth 3-manifold $M$ together  with an (ordered) pair of smooth line fields $L_1, L_2 \subset TM$, spanning a contact distribution $D=L_1\oplus L_2$. The path geometry dual to $(M,L_1,L_2)$ is  $(M, L_2, L_1)$.
\end{definition}

\begin{remark}
Another common name for  $(M,L_1, L_2)$ is a  {\em para-CR structure}, due to the formal similarity with a (Levi--non-degenerate) CR structure $(M,D,J)$. The latter is a contact distribution $D$ on a 3-manifold $M$ together with a complex structure $J \in \operatorname{End}(D)$, i.e. $J^2=-\operatorname{id}_D$; equivalently, it is a splitting  $D\otimes\C=D^{1,0}\oplus  D^{0,1}$, the direct sum of a conjugate pair of complex line bundles (the $\pm i$-eigenbundles of $J\otimes\C$).

In the real-analytic setting, CR and para-CR structures have a common complexification: a complex 3-manifold together with a pair of (complex) line fields spanning a (complex) contact distribution.
\end{remark}

\begin{remark}\label{remark:generalized-path-geometry}
Some authors define a path geometry as  a 2-parameter family of  curves on a surface $\Sigma$, a unique curve through any given point of $\Sigma$ in any given direction (see, e.g. the first paragraph of \cite{douglas},  or the ``fancy formulation" of Section 8.6 of \cite{IL}).  Definition \ref{def:pg} is  more precise and general: first, the surface $\Sigma$ is recovered from $(M, L_1, L_2)$ as the space of integral curves of $L_1$, which may exist as a smooth surface only locally.  Second, even if $\Sigma$ exists, the set of directions at a given   $x\in\Sigma$ for which
a curve exists may be only an open subset in $\P T_x\Sigma$. For example, for the path geometry of central ellipses in $\R^2\setminus \{0\}$ a curve exists only in non-radial directions.
Third, there may be more then one curve in a given direction. For example, for circles of fixed radius in $\R^2$,  there are {\em two}  circles passing through each point in a given direction. This can be  remedied by considering instead {\em oriented} circles of fixed radius and the {\em spherized} tangent bundle  $\S T \R^2$ ($T\R^2$, with the zero section removed, mod $\R^+$) instead of $\P T \R^2$. An analogous remedy applies to the aforementioned path geometry of horocycles in the hyperbolic plane.

We shall not dwell here further on  these details and refer the interested reader to Sections 4.2.3 and 4.4.3 of \cite{CS}, where our notion of a path geometry on a surface is called both a \textit{generalized path geometry} and a \textit{Lagrangean contact structure} on a $3$-manifold; the two notions differ in higher dimension.
\end{remark}

\subsection{Chains of path geometries via the Fefferman metric} In the CR case there is a well-known, naturally associated 4-parameter family of curves on $M$, called  {\em chains}, one chain for each given point in $M$ in a given  direction  transverse to the contact distribution. They are considered the CR analog of geodesics in Riemannian geometry
(see the recent article \cite{CMMM} for a variational formulation).  Chains were introduced  by \'E. Cartan while solving the equivalence problem of CR geometry \cite{C, J} and were studied  extensively by many authors, such as  Chern-Moser \cite{CM} and C. Fefferman \cite{F}, who showed that they arise from a natural construction, considerably simpler than Cartan's, nowadays called the {\em Fefferman  metric}: a conformal Lorentzian metric, i.e. of signature $(3,1)$, defined on the total space of a certain circle bundle over $M$. 
The chains of the CR structure are then the projections onto $M$ of the null geodesics of the Fefferman metric. 

Similarly, to each path geometry $(M,L_1, L_2)$ one can associate a natural 4-parameter family of curves on $M$, a unique curve through any  given point in $M$ in any given  direction  transverse to the contact distribution $D:=L_1\oplus L_2$.  The study of this natural class of curves is quite recent.  The earliest reference we know of is  a 2005 article of A. \v{C}ap and  V. \v{Z}\'adn\'ik \cite{CZ} (path geometries on surfaces appear there in Section 2 as $3$-dimensional \textit{Lagrangean contact structures}). See also Sections 5.3.7--8 and 5.3.13--14 of \cite{CS}. Both references define chains using  the associated Cartan geometry. However, as in CR case, there is a significant shortcut via the  Fefferman metric. This is a conformal metric of signature  $(2,2)$  on the total space of an $\R^*$-bundle  over $M$, and   the chains  are the projections onto $M$ of   non-vertical null geodesics of the  Fefferman metric.  In  this article we explain this construction and use it to give several concrete examples.

\begin{remark} As mentioned in Remark \ref{remark:generalized-path-geometry}, path geometries  on surfaces generalize in higher dimension to either (generalized) path geometries or   Lagrangean contact geometries. The Fefferman-type construction of  a conformal structure described here  generalizes in higher dimensions to Lagrangean contact structures but not to path geometries.
\end{remark}

\subsection{Contents of the article} In the next section we re-derive, as a warm-up and a reminder, the Fefferman   metric for a CR structure $(M,D,J)$. The construction appeared first in Fefferman's article \cite{F} for a CR manifold embedded as a real hypersurface in a complex manifold, followed by intrinsic  constructions, first direct ones in \cite{Fa, L}, then more advanced constructions that use the full solution of the equivalence problem for CR structures (Cartan bundle and connection), such as \cite{BDS,CZ,NS}.
%
We view instead a CR structure as a conformal class of sub-Riemannian geometries of contact type, solve the  equivalence problem of sub-Riemannian geometries of contact type following \cite{H}---which is much simpler than that for CR geometry, use a sub-Riemannian metric on $D$  to define a Lorentzian metric on $\S D$ (the spherization of $D$), then show that conformally equivalent sub-Riemannian metrics on $D$ induce conformally equivalent Lorentzian metrics on $\S D$. In retrospect, our construction can be regarded as a shorter version of \cite{Fa, L}, using  \cite{H}. It is still too complicated conceptually for our taste, and the below formula \eqref{eq:fef} for the metric appears a bit like magic, but this method is the best we have so far and is quite easy to work with.

Once the construction of Fefferman metric for CR geometry is understood, we construct in Section \ref{sect:pCR} in a similar fashion the Fefferman metric for a path geometry. As far as we know, our derivation is new, and before this article the only available construction of the Fefferman metric for path  geometry has been via the solution of the full equivalence problem for such a structure (see, e.g., \cite{CZ, NS}), which is considerably more involved than our derivation.

In Section \ref{sect:proj_chains} we prove the following  theorem, apparently new:
\begin{theorem}\label{thm:proj_chains}
 A path geometry on a surface $\Sigma$ is  projective if and only if the chains on $\P T\Sigma$ project to the paths in $\Sigma$.
 \end{theorem}

In the last section we study in some detail the chains of  four  homogeneous path geometries mentioned above:  straight lines, circles of fixed radius,  central ellipses of fixed area and horocycles in the hyperbolic plane.

\paragraph{Acknowledgments.} GB acknowledges support from CONACYT Grant A1-S-4588. TW is grateful for support and hospitality from CIMAT during an extended visit in the 2019--20 academic year and for support from Guilford College.

\section{The Fefferman metric for CR 3-manifolds (revisited)}
Let $(M,D,J)$ be a CR 3-manifold, i.e.  $D\subset TM$ is a contact 2-distribution (that is, $[D,D]=TM$) and $J\in \End(D)$ satisfies $J^2=-\operatorname{id}_D$.  Canonically associated to the CR structure is the circle bundle $ \S D\to M$ ($D$ with the zero section removes, mod $\R^+$) and a   conformal class of metrics of signature $(3,1)$ on  $\S D$,   the {\em Fefferman  metric}. It depends on the {\em second-order jet} of the CR structure, so is not so easy to see. The fibers of $\S D\to M$ are null geodesics, and the projections of the non-vertical null geodesics to $M$ are the {\em chains} of the CR structure, forming a 4 parameter family of curves on $M$.




\paragraph{The construction.} Fix a positive contact form $\eta^3$ on $M$, i.e. a 1-form satisfying
\begin{align}
 & D=\Ker(\eta^3) , \label{eq:contact}\\
 & \d\eta^3(X,JX)> 0\mbox { for every }X\in D,\ X\neq 0.\label{eq:positive}
\end{align}

\begin{remark}A general contact manifold does not admit necessarily a global contact form  (a 1-form whose kernel is $D$) but the contact structure of a CR manifold does, using the orientation of   $D$ induced by $J$. If $M$ is connected then any global  contact form is either positive or negative.
\end{remark}
Recall that the coframe bundle $\pi: F^*\to M$ is the principal $\GL_3(\R)$-bundle whose fiber at a
point $x\in M$ consists of all linear isomorphisms $u:T_xM\to \R^3$. 
The {\em tautological 1-form} on
$F^*$ is the  $\R^3$-valued 1-form  $\omega
$ whose value at
$u\in F^*$ is $u\circ(\d\pi)_u.$

Now a positive contact form $\eta^3$ on $M$ defines a positive-definite inner product on
$D$, $\<X, Y\>:=\d\eta^3(X,JY) $. An  {\em adapted coframe} is an extension of $\eta^3$ to a coframe $\eta=(\eta^1, \eta^2, \eta^3)^t$  (we view elements of $\R^3$ as {\em column} vectors),   satisfying
\begin{align}
 &\   \d\eta^3
        = \eta^1\wedge\eta^2\\
&   \<\cdot,\cdot\>= \left.\left[(\eta^1)^2+(\eta^2)^2\right]\right\vert_D.\label{eq:red}
 \end{align}

It is easy to show that for a fixed $\eta^3$ these 2 equations define a circle's worth of coframes at each $x\in M$. Thus,   let $S^1\subset\GL_3(\R)$ be the set of matrices of the form
%
 $$\left(\begin{array}{ccc}
\cos\varphi&\sin\varphi&0\\
\sin\varphi&\cos\varphi&0\\
0&0&1
\end{array}
\right),
$$
and $B\subset F^*$ the set of coframes adapted to $\eta^3$. Then
$B\to M$ is a principal  $S^1$-subbundle, an $S^1$-reduction of $F^*$,  whose local sections consist of adapted coframes.

We continue to denote by  $\omega=(\omega^1,\omega^2,\omega^3)^t$ the restriction of the tautological 1-form on $F^*$  to $B$. Then there are  unique 1-form $\alpha$ and functions $a_1, a_2$ on $B$  such that
\begin{align}\label{eq:str1}
\d
\left(
\begin{array}{c}
\omega^1\\
\omega^2\\
\omega^3
\end{array}
\right)
=-
\left(
\begin{array}{ccc}
0&\alpha&0\\
-\alpha&0&0\\
0&0&0
\end{array}
\right)
\wedge
\left(
\begin{array}{c}
\omega^1\\ \omega^2\\ \omega^3
\end{array}
\right)
+
\left(
\begin{array}{ccc}
a_1&a_2&0\\
a_2&-a_1&0\\
0&0&1
\end{array}
\right)
\left(
\begin{array}{c}
\omega^2\wedge \omega^3\\
\omega^3\wedge \omega^1\\
\omega^1\wedge \omega^2
\end{array}
\right).
\end{align}
(See equation  (1) of \cite{H}).
Furthermore, there are unique functions $b_1, b_2, K$ on $B$ such that
 \be\label{eq:str2}
 \d\alpha=b_1\omega^2\wedge \omega^3+b_2\omega^3\wedge \omega^1+K\omega^1\wedge\omega^2.
\ee
(See equation (4) of \cite{H}; in fact, $K$ descends to $M$. Also, $\alpha$ is essentially the Webster connection form \cite{W}, $a_1, a_2$ its torsion, and $K$ the Webster scalar curvature).

Define a  Lorentzian metric on $B$  by
\be \label{eq:met}
    \grm := \omega^1\cdot\omega^1 +\omega^2\cdot\omega^2+\omega^3\cdot \sigma,
\ee
where $\sigma$ is a $1$-form, to be determined later, and $\cdot$ is the symmetric product of 1-forms.

Let $\S D$ be the `spherization' (or `ray projectivization') of $D$, the quotient of $D$, with the zero section removed, by the dilation action of $\R^+$. There is an obvious $S^1$-action on $D$, commuting with the $\R^+$ action, thus making  $\S D$ a principal $S^1$-bundle. Note that $\S D$, unlike $B$, is canonically associated to $(M,D,J)$: to define $B$ we needed to choose the positive contact form $\eta^3$. Define an isomorphism of principal $S^1$-bundles
\be\label{eq:iso}
    h:B\to \S D, \quad u \mapsto [u^{-1}\e_1],
\ee
where $\e_1=(1,0,0)^t.$ That is, $h(u)=[X]\in \S D,$ where $X\in D$ is  the unique vector in  $T_xM$, $x=\pi(u)$,  satisfying $u^1(X)=1, u^2(X)=u^3(X)=0$. 
We then use  $h$ to map  the  Lorentzian metric  on $B$ of equation \eqref{eq:met} to   a Lorentzian metric   on  $\S D$. In general, for arbitrary $\sigma$ in formula \eqref{eq:met}, the  resulting  metric  on  $\S D$ depends  on the choice of $\eta^3$ in a complicated way, but for a careful choice of  $\sigma$  {\em  the   conformal class of the Lorentzian  metric on $\S D$  is  independent of the choice of $\eta^3$}.

\begin{remark} There are other models for the underlying space of the  Fefferman metric instead of
$\S D$ (a matter of taste). For example, one can take the spherization of the dual bundle $D^\vee$, in which case the formula for the identification $B\to \S D^\vee$ is a little simpler: $u\mapsto [u^1|_D].$ Another model is the spherization of the canonical bundle $\Lambda^{2,0}D\subset \Lambda^2T^*D\otimes\C$, as in \cite{L}; the identification with $B$  in this case is  $u\mapsto [u^3\wedge(u^1+iu^2)|_D].$ Also, the metric on $\S D$ is invariant under the antipodal map in each fiber (a circle), and so it descends to the (full) projectivization $\P D$.
\end{remark}
\begin{theorem} \label{thm:main}Let $(M,D,J)$ be a CR 3-manifold, $\S D\to M$ the spherization of $D$ and $\eta^3$ any positive contact $1$-form, as in equations \eqref{eq:contact} and \eqref{eq:positive}. Define a $1$-form $\sigma$ on the total space of the associated  circle bundle $B\to M$,
\be\label{eq:fef}\sigma={4\over 3}\alpha-{1\over 3}K\omega^3,
\ee
where $\alpha, K$ are defined via equations \eqref{eq:str1} and \eqref{eq:str2}. Then  the conformal class of the Lorentzian metric  induced on  $\S D$ by equation \eqref{eq:met}, via the isomorphism \eqref{eq:iso},  is independent of the choice of $\eta^3$. In fact, multiplying $\eta^3$ by a positive function  rescales  the induced metric on $\S D$ by the same factor.
\end{theorem}
\begin{proof}
If $\eta^3$ is a positive contact form on $M$, then any other positive contact form is of the form $\tilde\eta^3=\lambda^2\eta^3$, for some positive function $\lambda:M\to \R^+$. Changing $\eta^3$ to $\tilde\eta^3$ changes $B$ to $\tilde B$, another $S^1$-reduction of the coframe bundle of $M$, with corresponding metric $\tilde \grm$ and isomorphism $\tilde h:\tilde B\to S$. We thus need to show that the composition  $f:=\tilde h^{-1}\circ h : B\to \tilde B$ satisfies $f^*(\tilde g)=\lambda^2g.$

Let us pull-back  $\lambda$   to $B$ by the projection $B\to M$, denoting the result  by $\lambda$ as well. Then
\be\label{eq:lambda}
\d\lambda=\lambda_i\omega^i, \quad \d\lambda_i=\lambda_{i0}\alpha+\lambda_{ij}\omega^j,
\ee
 for some functions $\lambda_i, \lambda_{ij}, \lambda_{i0}$ on $B$, $1\leq i, j\leq 3$.  (Note that by definition $\lambda$ descends to $M$; in general the $\lambda_i$ do not, but $\lambda_3$ does.)
\begin{lemma}
    $$\lambda_{10}=-\lambda_2, \qquad \lambda_{20}=\lambda_1, \qquad \lambda_{12}-\lambda_{21}=\lambda_3 .$$
\end{lemma}
\begin{proof}
These identities follow immediately from expanding $\d(\d\lambda)=0.$
\end{proof}

 Now a section  $\eta=(\eta^1, \eta^2, \eta^3):M\to B$ of $B\to M$ is a coframe adapted to $\eta^3$,   so $f\circ\eta:M\to \tilde B$ is a section of $\tilde  B\to M$, a coframe adapted to $\tilde\eta^3=\lambda^2\eta^3.$

\begin{lemma}\label{lemma:cf}   $f\circ \etab=\Lambda\etab$,  where
$$\Lambda= \left(\begin{array}{ccr}
\lambda&0&-2\lambda_2\\
0&\lambda&2\lambda_1\\
0&0&\lambda^2
\end{array}
\right).
 $$
\end{lemma}
\begin{proof}
 It is enough to check  that $\tilde\etab:=\Lambda\etab$  satisfies equations \eqref{eq:contact}-\eqref{eq:red} above, with $\tilde\eta^3=\lambda^2\eta^3$ instead of $\eta^3$, as well  as $\tilde\eta^1(\tilde X)=1, \tilde\eta^2(\tilde X)=\tilde\eta^3(\tilde X)=0$  for $\tilde X=X/\lambda.$
\end{proof}

\begin{lemma}\label{lemma:pb} $f^*\omega_{\tilde B}= \Lambda\omega_B.$
\end{lemma}

 \begin{proof}Let $\eta\in B$, $\tilde\etab=f(\etab)$. By Lemma \ref{lemma:cf}, $\tilde\eta=\Lambda\etab,$  hence
 $(f^*\omega_{\tilde B})_\eta
 =(\omega_{\tilde B})_{\tilde\eta}\circ(\d f)_\etab
=\tilde\etab\circ(\d\tilde\pi)_{\tilde\etab}\circ(\d f)_\etab
=\tilde\etab\circ\d(\tilde\pi\circ f)_\etab=\tilde\etab\circ(\d\pi)_\etab
=\Lambda\etab\circ\d\pi_\etab
=\Lambda\omega_\etab.$
\end{proof}

\n{\bf Notation.} For sake of readability, we adopt henceforth the following abbreviated notation:
\[
\omega:= \omega_{B}, \  \grm:= \grm_B, \  \ldots, \  \tilde \omega:= f^*\omega_{\tilde B}, \ \tilde \grm:= f^*\grm_{\tilde B}, \  \ldots, \  \textrm{etc.}
\]
Thus, for example,  Lemma \ref{lemma:pb} reads $\tilde\omega=\Lambda\omega$.

\mn

We proceed with the proof of Theorem \ref{thm:main}.
It is clearly  enough to show an infinitesimal version of the claimed conformal invariance. Suppose $\lambda=\lambda(t)$ is differentiable and that it satisfies $\lambda(0)=1$. 
Denote by a dot the derivative with respect to $t$ at $t=0$ of objects on $\tilde B$ pulled back to  $B$ by $f$, e.g., $\dot \lambda=\lambda'(0), \ \dot\lambda_i=\lambda_i'(0),\ \dot\lambda_{ij}=\lambda_{ij}'(0)$, $\dot \grm=\left.{\d\over\d t}\right| _{t=0}\tilde \grm,$ etc. Then $\tilde \grm=\lambda^2 \grm$  if and only if $\dot \grm=2\dot\lambda\grm$ (for all  $\eta^3$ and $\lambda(t)$ satisfying $\lambda(0)=1$). Now calculate using the previous lemmas:
\begin{align*}
&\dot\omega^1= \dot\lambda\omega^1-2\dot\lambda_1\omega^3,
\  \   \dot\omega^2=\dot\lambda\omega^2+2\dot\lambda_2\omega^3, \ \
\dot\omega^3=2\dot\lambda\omega^3, \\
&\dot \grm=2\dot\lambda \grm+(\dot \sigma-4\dot\lambda_2\omega^1+4\dot\lambda_1\omega^2)\cdot \omega^3.
\end{align*}
Thus  $\dot\grm=2\dot\lambda\grm$ if and only if
\be\label{eq:conf}
\dot \sigma=4(\dot\lambda_2\omega^1-\dot\lambda_1\omega^2).
\ee
To calculate $\dot \sigma$, using formula \eqref{eq:fef}, we need formulas for $\dot\alpha$ and $\dot K.$ To find $\dot\alpha$ we find first a formula for $\tilde\alpha.$ Write the structure equations \eqref{eq:str1} for $\tilde\omega$, substitute $\tilde\omega=\Lambda\omega$, and equate coefficients. The result is $$ \tilde\alpha=\alpha+3{\lambda_2\over \lambda}\omega^1-3{\lambda_1\over \lambda}\omega^2-
\left[3{(\lambda_1)^2+(\lambda_2)^2
\over \lambda^2}+{\lambda_{11}+\lambda_{22}\over\lambda}\right]\omega^3.
$$
Taking derivative with respect to  $t$ at $t=0$ of the last formula, we get
$$\dot\alpha=3\dot\lambda_2\omega^1-3\dot\lambda_1\omega^2-(\dot\lambda_{11}+\dot\lambda_{22})\omega^3.$$
To find $\dot K$, there is a shortcut, avoiding an explicit formula for $\tilde K$, by noting first  that $K$ is defined by $\d\alpha\equiv K\omega^1\wedge\omega^2$ (mod $\alpha, \omega^3$). Taking $\d$ of the above formula for $\dot \alpha$, we  get,  using equations \eqref{eq:lambda},
 $\d\dot\alpha\equiv -4(\dot\lambda_{11}+\dot\lambda_{22})\omega^1\wedge\omega^2$ (mod $\alpha, \omega^3$). Taking derivative with respect to  $t$ of $\d\tilde\alpha\equiv\tilde K\tilde\omega^1\wedge\tilde\omega^2$ (mod $\tilde\alpha, \tilde\omega^3$), we get $\d\dot\alpha
\equiv  (\dot K+2\dot\lambda K) \omega^1\wedge\omega^2$ (mod $\alpha, \omega^3$),
hence
$$\dot K=-2\dot\lambda K-4(\dot\lambda_{11}+\dot\lambda_{22}).$$

Now let  $\tilde\sigma= c_1\tilde\alpha+c_2\tilde K\tilde\omega^3$, with some constants $c_1, c_2$. Then
\[
    \dot\sigma
        = c_1\dot\alpha+c_2(\dot K\omega^3+K\dot\omega^3)=3c_1(\dot\lambda_2\omega^1-\dot\lambda_1\omega^2)-(c_1+4c_2)(\dot\lambda_{11}+\dot\lambda_{22})\omega^3.
\]
Thus  equation  \eqref{eq:conf} is satisfied  if  $c_1=4/3, c_2=-1/3$.
\end{proof}

\subsection{ Example: left-invariant  CR structures  on $\SU_2$} The left-invariant $\su_2$-valued Maurer--Cartan form  on $\SUt$ is
\be\label{eqn:mc}\Theta=g^{-1}\d g= \left(\begin{array}{cc}
i\theta^1 &\theta^2+i\theta^3\\
-\theta^2+i\theta^3 & -i\theta^1
\end{array}\right).
\ee
 The Maurer--Cartan equation $\d\Theta=-\Theta\wedge\Theta$ gives
\be\label{eq:mc1}
 \d\theta^1=-2\theta^2\wedge\theta^3 , \qquad
 \d\theta^2=-2\theta^3\wedge\theta^1 , \qquad
 \d\theta^3=-2\theta^1\wedge\theta^2 .
\ee
For each $t\in[1,\infty)$ let
$$
    \eta^1 =  \sqrt{t} \,\theta^1           , \qquad
    \eta^2 =             \theta^2 /\sqrt{t} , \qquad
    \eta^3 = -           \theta^3 /2        .
$$
One can show that every left-invariant CR structure $D^{0,1}\subset T\SUt\otimes\C$ is equivalent (via right translation), for a unique $t\in[1,\infty)$,  to $\{\eta^1+i\eta^2,\eta^3\}^\perp.$ For $t=1$ we obtain the standard `spherical' CR structure on $\SU_2\simeq S^3$, for $t>1$ these are non-spherical CR structures, distinct $t$ determine inequivalent structures (see \cite{BJ}, Prop. 5.1).
 We use \eqref{eq:mc1} to find
$$
    \d\eta^1 = 4t \,\eta^2\wedge\eta^3 , \qquad
    \d\eta^2 = (4/t)\eta^3\wedge\eta^1 , \qquad
    \d\eta^3 =      \eta^1\wedge\eta^2 .
$$
Using this coframe we  identify  $B\simeq \SUt\times S^1$ and  $\omega=\bar u\cdot\eta$, where $u=e^{i\varphi}.$ Explicitly,
$$\omega^1=\sqrt{t}(\cos\theta)\theta^1+
{1\over\sqrt{t}}(\sin\theta)\theta^2,\ \omega^2=-\sqrt{t}(\sin\theta)\theta^1+{1\over\sqrt{t}}(\cos\theta)\theta^2,\ \omega^3=-{1\over 2}\theta^3.$$
Inserting these into equations \eqref{eq:str1}--\eqref{eq:str2}, we obtain
$$\alpha=\theta^4-\left(t+{1\over t}\right)\theta^3, \quad K=2\left(t+{1\over t}\right),$$
where $\theta^4:=\d\varphi$.
Inserting all this  into  equations \eqref{eq:met}--\eqref{eq:fef}, we get
$$
\grm=t\left(\theta^1\right)^2+{1\over t}\left(\theta^2\right)^2+{1\over 2}\left(t+{1\over t}\right)\left(\theta^3\right)^2-{2\over 3}\theta^3\cdot\theta^4.
$$
This is essentially formula (15) of \cite{CM}; the coefficient of our $\theta^3 \cdot \theta^4$ term can be made to agree with the cited formula by rescaling the $\varphi$ coordinate by a constant. See also \cite{CM} for a study of the chains of this example via null geodesics of  the Fefferman metric.


\section{The Fefferman metric for path geometries}\label{sect:pCR}

Let $(M,L_1, L_2)$  be a path geometry, i,e, $L_1, L_2$ is a pair of   line fields on a 3-manifold $M$,  spanning a contact distribution  $D:=L_1\oplus L_2$. Let us fix a contact form $\eta^3$, that is, $D=\Ker(\eta^3)$ (possibly defined only locally, see Remark \ref{remark:contact} below). An {\em adapted coframe}  (with respect to $\eta^3$) is an extension of $\eta^3$ to a (local) coframe $(\eta^1, \eta^2, \eta^3)$  satisfying
\begin{align}
&\d\eta^3=\eta^1\wedge\eta^2,\label{eq:adapted1}\\
&\eta^1|_{L_2}=\eta^2|_{L_1}=0. \label{eq:adapted2}
 \end{align}
These equations  define an $\R^*$-structure, i.e. an $\R^*$-principal subbundle $B\subset F^*$,  whose local sections are the coframes adapted to $\eta^3$, where $s\in \R^*$ acts by $(\eta^1, \eta^2, \eta^3)\mapsto( \eta^1/s, s\eta^2, \eta^3).$

Let $D^*=D\setminus (L_1\cup L_2)\subset D$,  with  spherization $\S D^*\subset \S D$. The Fefferman metric associated to the path geometry is a conformal pseudo-Riemannian metric of signature $(2,2)$ on $\S D^*$. We shall define it in a manner  similar to the  CR case. The splitting $D=L_1\oplus L_2$ defines an involution $J\in\End(D), $ $J^2=\id,$ by
\be\label{eq:J}
J(X_1+X_2)=X_1-X_2,\quad  X_i\in L_i, \ i=1,2.
\ee The contact form $\eta^3$ defines on $D$ an area form, $\left.\d\eta^3\right|_D$, and an indefinite metric of signature $(1,1)$, $\<X,Y\>:=\d\eta^3(X,JY).$ Now $D^*=D^+\cup D^-$, where $D^\pm$ are the positive (resp. negative) vectors with respect to $\<\cdot, \cdot\>$, and corresponding decomposition $\S D^*=\S D^+\cup \S D^-$. Both $\S D^\pm$ are $\R^*$-principal bundles over $M$, where $s\in\R^*$ acts by $[X_1+X_2]\mapsto [sX_1+X_2/s],$ $X_i\in L_i$.
Note that $D^\pm$ are interchanged by $J$ or by taking $-\eta^3$ instead of $\eta^3$.
There is an identification  of $\R^*$-principal bundles,
\be\label{eq:hiso}
h:B\to \S D^+, \ u\mapsto [X],\ \mbox{where}\ u^1(X)=u^2(X)=1, u^3(X)=0.
\ee

We shall  define a pseudo-Riemannian metric of signature $(2,2)$ on $B$,  map it by $h$ to $\S D^+$, then by $J$  to $\S D^-$. As in the CR case, we show that the associated conformal class of metrics on $\S D^*$ is independent of the chosen contact form $\eta^3$.

\begin{remark}\label{remark:contact} A general contact 3-manifold is naturally oriented. (Proof: choose a local contact form $\eta^3$, then $\eta^3\wedge\d\eta^3$ is a volume form; multiplying $\eta^3$ by a non-vanishing $\lambda$ multiplies this volume form  by $\lambda^2$, so does not change the associated orientation.) The Lie bracket of sections of $D$ defines an isomorphism $\Lambda^2(D)\to TM/D$, but these isomorphic  line bundles need not be trivial, i.e. there might not exist on $M$ a global contact form (a non-vanishing section of $D^\perp\simeq(TM/D)^*).$  In the CR case, $J$ defines an orientation on $D\subset TM$, hence of $TM/D$ (since $TM$ is oriented), and a dual orientation of $D^\perp=(TM/D)^*$, so there is always a global  contact form. This is not the case for a path geometry  (e.g., $M=\P T\R^2$, equipped with the standard flat path geometry).  But this topological difficulty is minor, we can still define $B$ locally, then show  that the conformal structures  defined on $\S D^*$ restricted to open subsets of $M$ agree  on intersections. We shall not dwell on the details.
\end{remark}

We shall now proceed with the plan outlined above, in the paragraph before Remark \ref{remark:contact}.

The structure equations for any $\R^*$-connection form $\alpha$ on   $B\to M$ are
\begin{align*}
\d
\left(
\begin{array}{c}
\omega^1\\
\omega^2\\
\omega^3
\end{array}
\right)
=-
\left(
\begin{array}{ccc}
\alpha&0&0\\
0&-\alpha&0\\
0&0&0
\end{array}
\right)
\wedge
\left(
\begin{array}{c}
\omega^1\\ \omega^2\\ \omega^3
\end{array}
\right)
+
\left(
\begin{array}{ccc}
T^1_{23}&T^1_{31}&T^1_{12}\\
T^2_{23}&T^2_{31}&T^2_{12}\\
0&0&1
\end{array}
\right)
\left(
\begin{array}{c}
\omega^2\wedge \omega^3\\
\omega^3\wedge \omega^1\\
\omega^1\wedge \omega^2
\end{array}
\right),
\end{align*}
where $T^i_{jk}$ are some real functions on $B$ (the  coefficients of the torsion tensor of the connection). Starting from any such connection it  is easy to show that it can be modified, in a unique way,  by adding to $\alpha$ multiples of the $\omega^i$, so as to render  $T^1_{31}=T^1_{12}=T^2_{12}=0$ (in fact doing so also solves the equivalence problem for path geometry equipped with a fixed contact form).  Taking the exterior derivative of  $\d\omega^3=\omega^1\wedge\omega^2$ shows that $T^2_{23}=0$ as well. The structure equations now become
\begin{align}\label{eq:str3}
\begin{split}
\d\omega^1&=-\alpha\wedge\omega^1+ a_1\omega^2\wedge \omega^3\\
\d\omega^2&=\phantom{-}\alpha\wedge\omega^2+ a_2\omega^3\wedge \omega^1\\
\d\omega^3&=\phantom{-}\omega^1\wedge \omega^2
\end{split}
\end{align}
for some functions $a_1, a_2$  on $B$. Taking exterior derivative of these equations we get
\begin{align}\label{eq:str4}\d\alpha=b_1\omega^2\wedge \omega^3+b_2\omega^3\wedge \omega^1 +K \omega^1\wedge \omega^2
\end{align}
for some functions $b_1, b_2, K$ on $B$ (i.e. $\d\alpha$ is semi-basic, containing  no $\alpha\wedge\omega^i $ terms.)

\begin{theorem} Let $(M,L_1, L_2)$ be a path geometry 
and $\S D^*\subset \S D$ the set of rays in $D=L_1\oplus L_2$ not contained in $L_1\cup L_2$. Then there is a canonically associated conformal class of metrics of signature $(2,2)$ on $\S D^*$, called the Fefferman metric, defined as follows. Associated with each   contact 1-form  $\eta^3$  on $M$ is an $\R^*$-reduction $B\to M$ of the coframe bundle of $M$, given by equations \eqref{eq:adapted1}-\eqref{eq:adapted2}, a unique $\R^*$-connection form $\alpha$ on $B$ satisfying equations \eqref{eq:str3} and the 1-form
\be\label{eq:feff2}
\sigma:=-{2\over 3}\alpha+{1\over 6}K\omega^3,
\ee
where $K$ is defined via equations \eqref{eq:str4}, and where $\omega^1,\omega^2,\omega^3$ are the tautological 1-forms on the coframe bundle of $M$ restricted to $B$. Then
\be\label{eq:feff3}\grm:=\omega^1\cdot\omega^2+\omega^3\cdot\sigma
\ee
is a pseudo-Riemannian metric on $B$ of  signature $(2,2)$. There is also associated with $\eta^3$ a decomposition $\S D^*=\S D^+\cup \,\S D^-$ and  $\R^*$-isomorphisms $h:B\to \S D^+$, $J\circ h: B \to D^-$,  where $h$ is given by equation \eqref{eq:hiso} and $J$ by equation \eqref{eq:J}, such that the conformal class of the induced metric on  $\S D^*$  is independent of the choice of $\eta^3$; in fact,  multiplying $\eta^3$ by a smooth non-vanishing function rescales the induced metric on $\S D^*$ by the same factor. 

\end{theorem}

\begin{proof} The proof is very similar to the CR case.  Here are the formulas that differ:
\begin{align*}
\lambda_{10} &= \lambda_1, \qquad
\lambda_{20}  =-\lambda_2, \qquad
\lambda_{12}-\lambda_{21} = \lambda_3,\\
\tilde\alpha &=\alpha+{3\lambda_1\over\lambda}\omega^1-{3\lambda_2\over\lambda}\omega^2-\left({\lambda_{12}+\lambda_{21}\over \lambda}+{6\lambda_1\lambda_2\over\lambda^2}\right)\omega^3,\\
\dot\alpha &=3\dot\lambda_1\omega^1-3\dot\lambda_2\omega^2-\left(\dot\lambda_{12}+\dot\lambda_{21}\right)\omega^3,\\
\d\dot\alpha&\equiv-4\left(\dot\lambda_{12}+\dot\lambda_{21}\right)\omega^1\wedge\omega^2\equiv\left(\dot K+2\dot\lambda K\right)\omega^1\wedge\omega^2\ \pmod{\alpha, \omega^3},\\
\dot K&=-2\dot\lambda K-4\left(\dot\lambda_{12}+\dot\lambda_{21}\right),\\
\dot\grm&=2\dot\lambda\grm+(\dot\sigma+2\dot\lambda_1\omega^1-2\dot\lambda_2\omega^2).
\end{align*}
\qedhere
\end{proof}

\begin{definition}A {\em chain} of a path geometry $(M,L_1, L_2)$ is the projection to $M$ of an unparametrized non-vertical null geodesic of the associated Fefferman conformal metric on $\S D^*$.
\end{definition}

\begin{proposition}
~
\begin{enumerate}

\item The $\R^*$-action on $\S D^*$ is by conformal isometries.

\item  For every point in $M$, in every given direction  transverse to $D$, there is a unique chain passing through this point in the given direction.

\item  The fibers of $\S D^*\to M$ are null geodesics and project to constant curves on $M$. 

\end{enumerate}
\end{proposition}

\begin{proof}

\mn (1) For every contact form $\eta^3$, the map $h:B\to \S D^+$ (by definition, a conformal isometry) is $\R^*$-equivariant, hence it is enough to verify that the pseudo-Riemannian metric on $B$ given by equations \eqref{eq:feff2}--\eqref{eq:feff3} is $\R^*$-invariant. This follows from the $\R^*$-invariance of   $\alpha,\omega^3, K$ and the  $\R^*$-equivariance  $R_s^*\omega^1=\omega^1/s,$ $R_s^*\omega^2=s\omega^2$.

\mn (2) Let $x\in M$ and $v\in T_xM,$ $v\not\in D_x$. Pick a contact form $\eta^3$ and work on the
 associated bundle $B$. The fiber $B_x$ consists of coframes $u=(u^1, u^2, u^3)$ on
  $T_xM$ adapted to $\eta^3$, as in equations \eqref{eq:adapted1}--\eqref{eq:adapted2}. We show that for every $u\in B_x$ there is  a unique lift  $\tilde v\in T_u B$ of $v$ which is null. Now $\tilde v$ is a
   lift of $v$ if and only if $\omega^i(\tilde v) =u^i( v)$, $i=1,2,3$. It remains to determine $\sigma(\tilde v)$.  Now  $\omega^3(\tilde v) =u^3( v)\neq 0$
and, by formula \eqref{eq:feff3},  $\tilde v$ is null if and only if $\omega^1(\tilde v)\omega^2(\tilde v)+\omega^3(\tilde v)\sigma(\tilde v)=0,$  i.e. $\sigma(\tilde v)=-u^1(v)u^2(v)/u^3(v).$
   This shows that $v\in T_xM$ has a unique null lift at $u\in B_x.$
   The null geodesic through $u$ in the direction of $\tilde v$ projects to a chain through $x$ in the direction of $v$.
   This proves existence of the required chain.
   As for uniqueness, we need to show that if we repeat the above at another point of $B_x$,
   say $s\cdot u\in B_x$, we obtain the same chain. We use the fact that $s$ acts on $B$
   by isometries $R_s$, mapping $\tilde v$ to the unique null-lift of $v$ at $s\cdot u$,
   and the null geodesic through $u$ tangent to $\tilde v$ to the null geodesic through $s\cdot u$ in the direction of $(R_s)_* \tilde v$. Since $R_s$ commutes with the projection $B\to M$, the two null geodesics project to the same chain in $M$.

\mn (3) The vertical distribution of  $B\to M$ is given by  $\omega^1 = \omega^2 = \omega^3 = 0$, thus $\grm=\omega^1 \cdot \omega^2 + \omega^3 \cdot \sigma$ restricted to  the fibers vanishes, so these fibers are null curves. We proceed to show that they are null geodesics.

As shown in part (1) above, the principal $\R^*$-action on $B$ is isometric. Let $\zeta$ denote an infinitesimal generator of this action (i.e. a nonzero vertical null Killing vector field on $B$).  The fibers of $B\to M$ are the integral curves of $\zeta$, hence to show that these fibers are null geodesics it is enough to show that  $\nabla_\zeta \zeta = 0$, or in index notation,
\[
    \zeta^b \nabla_b \zeta^a = 0 .
\]
Lowering an index of $\nabla_b \zeta^a$ (using $\grm$), splitting $\nabla_b \zeta_a$ into its symmetric and antisymmetric parts, and contracting with $\zeta^b$ gives
\be\label{eq:zeta}
    \zeta^b \nabla_b \zeta_a = \zeta^b \cdot \frac{1}{2}(\nabla_a \zeta_b + \nabla_b \zeta_a) + \zeta^b \cdot \frac{1}{2}(\nabla_a \zeta_b - \nabla_b \zeta_a) .
\ee
The quantity $\frac{1}{2} (\nabla_a \zeta_b + \nabla_b \zeta_a)$ in the first term is $(\mathcal L_\zeta g)_{ab}$, but, per part (1), $\grm$ is $\zeta$-invariant---that is, $\mathcal L_\xi g = 0$---and so the first term vanishes.

The quantity $\frac{1}{2} (\nabla_a \zeta_b - \nabla_b \zeta_a)$ in the second term is $(d\zeta^\flat)_{ab}$, so the second term is $-\iota_{\zeta} (d\zeta^\flat)$, where $\iota_{\zeta}$ denotes interior multiplication by $\zeta$. Since $\zeta$ generates the $\R^*$-action on $B \to M$ and $\alpha$ is a connection form thereon, $\alpha(\zeta)$ is a nonzero constant, and by rescaling $\zeta$ by a nonzero constant we may as well assume  $\alpha(\zeta) = 3$. Lowering an index with $\grm$ (equations \eqref{eq:feff2}--\eqref{eq:feff3}) then gives  $\zeta^\flat =   -\omega^3$, so the third equation of \eqref{eq:str3} yields $-\iota_{\zeta} (\d\zeta^\flat) = \iota_{\zeta} \d  \omega^3 = \iota_{\zeta} (\omega^1 \wedge \omega^2) = 0$.
\end{proof}

\begin{remark}
In fact, chains come equipped with a preferred projective structure (see, e.g., \cite[Theorem 5.3.7]{CS}, which applies to all so-called parabolic contact structures), but we do not need that structure here.
\end{remark}

\subsection{Chains of $y''=f(x,y,y')$}

Here $\Sigma=J^0(\R,\R)=\R^2$, with coordinates $(x,y)$, and $M=J^1(\R,\R)=\R^3$, with coordinates $(x,y,p)$ and contact distribution $D=\Ker(\d y-p\, \d x).$  The paths in $\Sigma$ are graphs of solutions $y(x)$ to $y''=f(x,y,y')$, and their lifts to $M$ are  graphs of their first jets,  $(x, y(x))\mapsto (x, y(x), y'(x))$.
So here %
$$L_1=\Span \left(\partial_p\right),\quad  L_2=\Span \left[\partial_x+p\partial_y+f(x,y,p)\partial_p\right].$$
We fix the contact form $\eta^3:=\d y- p\,  \d x$. An adapted coframe on $M$,  satisfying equations \eqref{eq:adapted1}--\eqref{eq:adapted2},  is
\be\label{eq:etas}
\eta^1:=\d p-f\d x, \quad\eta^2:=-\d x, \quad \eta^3:=\d y-p\,\d x.
\ee
Any other adapted coframe is of the form $s\cdot\eta=(\eta^1/s, s\eta^2, \eta^3)^t$, $s:M\to\R^*$. This defines an identification $\R^3\times \R^*\to B$, $(x,y,p,s)\mapsto s\cdot\eta(x,y,p).$ Under this identification,
\be\label{eq:omegas}
\omega^1=\eta^1/s, \quad \omega^2=s\eta^2, \quad \omega^3=\eta^3.
\ee

The following proposition was proved in \cite[equation (31)]{NS} by solving the full equivalence problem for path geometry.

\begin{proposition} The Fefferman metric on $B=J^1(\R,\R)\times \R^*$ is
 \be\label{eq:fef22}\grm=-\d x\cdot (\d p-f \d x ) +
 {1\over 6} (\d y - p\,\d x) \cdot \left[4 f_p \d x + f_{pp}(\d y  - p\,\d x )-4 \d\tau\right],
\ee
 where $\d\tau=\d s/s. $
\end{proposition}

\begin{proof}
Solving equations \eqref{eq:str3}--\eqref{eq:str4}, with $\omega^i$ given by equations \eqref{eq:etas}--\eqref{eq:omegas}, we obtain
$$\alpha=-f_p\d x+\d\tau ,\  K=f_{pp} \ \Longrightarrow\  \sigma={2\over 3}f_p\d x +{1\over 6}f_{pp}(\d y - p \d x)-{2\over 3}{\d \tau}.
$$
Using this in equations \eqref{eq:feff3}-\eqref{eq:feff3} gives the claimed formula.
\end{proof}

\begin{proposition}\label{prop:yp}
The chains of the path geometry corresponding to a 2nd order  ODE $y''=f(x,y,y')$ are the curves in $J^1(\R,\R)$ which are the graphs of solutions $(y(x), p(x))$ of  the system
\begin{align}\label{eq:yp1}
y''&=  f+ f_p \Delta+ \frac{1}{2}f_{pp} \Delta^2 + {1\over 6}f_{ppp}  \Delta^3\\
\begin{split}p''&=-\frac{2 (p'-f)^2}\Delta +f_p(3 p'-2 f )+f_x+pf_y+\left[ f_{pp}(p' -f)+2f_y\right]\Delta  \\
&+\frac{1}{6}\left[ f_{ppp} (p'-2f)   -f_{xpp}+4f_{yp}-pf_{ypp} ) \right] \Delta ^2, \label{eq:yp2}
\end{split}
\end{align}
where $\Delta=y'-p.$
\end{proposition}

\begin{proof} Using the metric \eqref{eq:fef22}, we write the geodesic equations on $B$,
$$\begin{aligned}
\ddot x&=\frac{1}{6} \left[\left(p \dot{x}-\dot{y}\right) \left(f_{{ppp}} \left(\dot{y}-p \dot{x}\right)+2 \dot{x} f_{{pp}}\right)-2 \dot{x}^2 f_p-4 \dot{\tau } \dot{x}\right],\\
\ddot y &=\frac{1}{6} \left[2 \dot{x} \left(p f_{{pp}} \left(p \dot{x}-\dot{y}\right)-p \dot{x} f_p-2 p \dot{\tau }+3 \dot{p}\right)-p f_{{ppp}} \left(\dot{y}-p \dot{x}\right)^2\right],\\
\ddot p &=\frac{1}{6} \left[
-p^3 \dot{x}^2 f_{{ypp}}+2 p^2 \dot{x} \dot{y} f_{{ypp}}+4 p^2 \dot{x}^2 f_{{yp}}
\right.\\ &\left.\qquad\quad
-2 f \left(\left(\dot{y}-p \dot{x}\right) \left(f_{{ppp}} \left(\dot{y}-p \dot{x}\right)+2 \dot{x} f_{{pp}}\right)+2 \dot{x}^2 f_p+4 \dot{\tau } \dot{x}\right)
\right.\\ &\qquad\quad\left.
+ 2 \dot{p}f_{{pp}}( \dot{y} -  p \dot{x})-f_{{xpp}} \left(\dot{y}-p \dot{x}\right)^2-8 p \dot{x} \dot{y} f_{{yp}}-6 p \dot{x}^2 f_y
\right.\\ &\qquad\quad\left.
+8 \dot{p} \dot{x} f_p-p \dot{y}^2 f_{{ypp}}+12 \dot{x} \dot{y} f_y+6 \dot{x}^2 f_x+4 \dot{y}^2 f_{{yp}}+4 \dot{p} \dot{\tau }\right].
\end{aligned}
$$
(We do not need the $\tau$ equation.) Next use   formula  \eqref{eq:fef22} and the nullity condition to eliminate $\dot \tau$ from the equations,
$$
\dot{\tau }={1\over 4} \left(f_{{pp}} \left(\dot{y}-p \dot{x}\right)+4 \dot{x} f_p\right)
- {3\over 2}{ \dot{x}(\dot{p}-f \dot{x})\over  (\dot{y}-p \dot{x}) },
$$
($\tau$ itself does not appear explicitly, because of the $\R^*$-invariance of the metric). Then substitute into $ y''=(\ddot y\dot x-\ddot x \dot y)/\dot x^3, p''=(\ddot p\dot x-\ddot x \dot p)/\dot x^3$ the expressions for $\ddot x, \ddot y, \ddot p$ from the geodesic equations, and finally make the substitutions  $\dot y=\dot x y', \dot p=\dot x p'$ to obtain the desired equations (all instances of $\dot x$ cancel out because the geodesic equation is homogeneously quadratic in velocities).
\end{proof}

\subsection{Chains of projective path geometries}\label{sect:proj_chains}

Here we prove Theorem \ref{thm:proj_chains}, which was announced in the introduction. Recall that, by definition,  a path geometry is projective if  the paths are the (unparametrized) geodesics of a torsion-free affine connection.

\mn{\bf Theorem \ref{thm:proj_chains}}. {\em A path geometry on a 2-dimensional manifold $\Sigma$ is  projective   if and only if all chains on $\P T\Sigma$ project to the paths in $\Sigma$.}
\begin{proof}
This is a local statement so we can assume without loss of generality the situation  studied in the previous subsection, i.e.    the paths  are given in the $xy$ plane by  graphs of solutions $y(x)$ of $  y'' = f(x, y, y')$ for some smooth $f$,  and  the associated chains in $xyp$-space are  the graphs of solutions $(y(x), p(x))$ to  the chain equations \eqref{eq:yp1}--\eqref{eq:yp2} of Proposition \ref{prop:yp}. As is well known, such a path geometry is projective if and only if $f(x, y, p)$ is a polynomial in $p$ of degree at most 3 (see \cite{CartanProjective}, also Section 4 of \cite{D}).
 The statement we are to prove therefore reduces to the following lemma:

\begin{lemma}
Every solution $(y(x), p(x))$ of equations \eqref{eq:yp1}--\eqref{eq:yp2}  satisfies $y''=f(x,y,y')$ if and only if $f(x,y,p)$ is polynomial in $p$ of degree at most $3$.
\end{lemma}

We proceed with the proof of the lemma. Assume $f(x,y,p)$ is polynomial  in $p$ of degree $\leq 3$. Then $f(x,y,y')$ is given by the  cubic Taylor polynomial of $f$ with respect to $p$:
\be\label{eq:tay}
    f(x,y,y')=f+f_p(y'-p)+\frac{1}{2}f_{pp}(y'-p)^2+ \frac{1}{6} f_{ppp} (y' - p)^3,
\ee
where $f$ and its derivatives on the right hand side are evaluated at $(x,y,p)$. Now the right hand side of the last equation, evaluated at $y=y(x), y'=y'(x), p=p(x)$,  is the right hand side of the chain equation \eqref{eq:yp1}. It follows that if $(y(x), p(x))$ satisfy equations  \eqref{eq:yp1}--\eqref{eq:yp2} then $y(x)$ satisfies $y''(x)=f(x,y(x),y'(x)),$ as needed.

Conversely,  suppose $f(x,y,p)$ is not polynomial in $p$ of degree $\leq 3$. Then there is a neighborhood $U\subset \R^3$ such that for all $(x,y,p), (x,y, y')\in U$, with $y'\neq p$, equation \eqref{eq:tay} does not hold. It follows that the chains in this neighborhood do not project to solutions of $y''=f(x,y,y')$. 
\end{proof}

\begin{remark}
One should also be able to prove Theorem \ref{thm:proj_chains} using the general machinery of parabolic geometry concerning correspondence spaces \cite[Section 4.4]{CS} and canonical curves \cite[Section 5.3]{CS} in a way that may be readily generalizable to other types of parabolic geometries and families of curves. Such a proof would take us too far afield here, so we will take up this approach elsewhere.
\end{remark}

\section{Examples of path geometries and their chains}
In this section we illustrate the general theory of the previous section by determining explicitly the chains of  some  {\em homogeneous} path geometries. First, the flat path geometry on $\RP^2$, admitting an 8-dimensional symmetry group, then 3 of the items  of Tresse's classification \cite{T} of `submaximal' path geometries, i.e. those  admitting a 3-dimensional  group of symmetries. In each case we exploit the symmetry to reduce the chain equations to determining null geodesics on a group with respect to a left-invariant sub-Riemannian metric. Then a well-known procedure reduces the equation to the Euler equations on the dual of the Lie algebra of the group and are integrable.
\subsection{ The flat path geometry}
Here $M\subset \RP^2\times(\RP^2)^*$ is the set of incident point-line pairs $(q,\ell)$ (equivalently,  the manifold $F_{1,2}$ of full flags in  $\R^3$) and  $L_1, L_2\subset TM$ are  the tangents to the  fibers of the projections onto  the first and second factor  (respectively).
\begin{proposition}\label{prop:flat}
For each non-incident pair $(q_*, \ell_*)\in  \RP^2\times(\RP^2)^*\setminus M$ consider the set  of incident pairs $(q, \ell)\in M$ such that $q\in\ell_*, q_*\in\ell.$ This is a chain in $M$ and all chains in $M$ are of this form. See Figure \ref{fig:flat_chains}.
\end{proposition}

\begin{figure}
\centering
\def\svgwidth{.3\textwidth}
\begingroup%
  \makeatletter%
  \providecommand\color[2][]{%
    \errmessage{(Inkscape) Color is used for the text in Inkscape, but the package 'color.sty' is not loaded}%
    \renewcommand\color[2][]{}%
  }%
  \providecommand\transparent[1]{%
    \errmessage{(Inkscape) Transparency is used (non-zero) for the text in Inkscape, but the package 'transparent.sty' is not loaded}%
    \renewcommand\transparent[1]{}%
  }%
  \providecommand\rotatebox[2]{#2}%
  \newcommand*\fsize{\dimexpr\f@size pt\relax}%
  \newcommand*\lineheight[1]{\fontsize{\fsize}{#1\fsize}\selectfont}%
  \ifx\svgwidth\undefined%
    \setlength{\unitlength}{163.51001358bp}%
    \ifx\svgscale\undefined%
      \relax%
    \else%
      \setlength{\unitlength}{\unitlength * \real{\svgscale}}%
    \fi%
  \else%
    \setlength{\unitlength}{\svgwidth}%
  \fi%
  \global\let\svgwidth\undefined%
  \global\let\svgscale\undefined%
  \makeatother%
  \begin{picture}(1,0.89112681)%
    \lineheight{1}%
    \setlength\tabcolsep{0pt}%
    \put(0,0){\includegraphics[width=\unitlength,page=1]{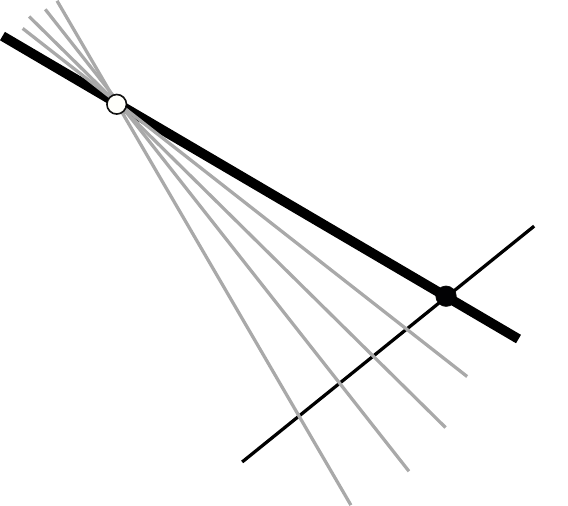}}%
    \put(0.74649663,0.45081735){\color[rgb]{0,0,0}\makebox(0,0)[lt]{\lineheight{0}\smash{\begin{tabular}[t]{l}$q$\end{tabular}}}}%
    \put(0.94210671,0.21983903){\color[rgb]{0,0,0}\makebox(0,0)[lt]{\lineheight{0}\smash{\begin{tabular}[t]{l}$\ell$\end{tabular}}}}%
    \put(0.2415674,0.74983574){\color[rgb]{0,0,0}\makebox(0,0)[lt]{\lineheight{0}\smash{\begin{tabular}[t]{l}$q_*$\end{tabular}}}}%
    \put(0.96001188,0.44007388){\color[rgb]{0,0,0}\makebox(0,0)[lt]{\lineheight{0}\smash{\begin{tabular}[t]{l}$\ell_*$\end{tabular}}}}%
    \put(0,0){\includegraphics[width=\unitlength,page=2]{chains_flat.pdf}}%
  \end{picture}%
\endgroup%

\caption{Chains of the flat path geometry (straight lines). }\label{fig:flat_chains}
\end{figure}

To prove it, note that $\GL_3(\R)$ acts naturally on $(M, L_1, L_2)$. We look for a 3-dimensional subgroup $G\subset \GL_3(\R)$ acting on $M$ with an open orbit. Fixing a point $m_0=(q_0, \ell_0)\in M$ yields we get two left-invariant line fields  $L_1, L_2\subset TG,$ given by their value $(L_1)_\id, (L_2)_\id\subset \g$, the Lie algebras of the stabilizers of $q_0, \ell_0$ (resp.). It is then easy to find left-invariant  adapted coframes on $G$ describing $L_1, L_2$  and the associated Fefferman metric. We consider two such $G$: the Heisenberg group and the Euclidean group.

\subsubsection{First proof of Proposition \ref{prop:flat}: via the Heisenberg group}
 Let   $\mathrm H$ be the set  of matrices of the form

\be\label{eq:H}
\left(\begin{array}{ccc}1&z&y\\  0&1&x\\ 0&0&1\end{array}\right),
\quad x,y,z\in\R.
\ee
Its  Lie algebra  $\h$  consists of matrices of the form
\be\label{eq:Ee}
\left(\begin{array}{ccc}0&x^1&x^3\\  0&0&x^2\\ 0&0&0\end{array}\right)
, \quad x^i\in\R.
\ee
Let $\theta^i$ be the  left-invariant 1-form on $\mathrm H$ whose value at $\id\in \mathrm H$ is
$x^i$, $i=1,2,3$.  Then
$$\Theta:=
\left(\begin{array}{ccc}0&\theta^1&\theta^3\\  0&0&\theta^2\\ 0&0&0\end{array}\right)
$$
is the left-invariant Maurer--Cartan form on $\mathrm H$, satisfying $\d\Theta=-\Theta\wedge\Theta$,  from which we get
\be\label{eq:stre}
\d  \theta^1=\d \theta^2=0, \   \d \theta^3=-\theta^1\wedge \theta^2.
\ee

Identify  $\R^2$ with an  affine plane  in $\R^3$, $(x,y)\mapsto (y,x,1)$. It is $\mathrm H$-invariant, and the resulting affine action on  $\R^2$ is $(x_0,y_0)\mapsto (x_0+x, y_0+y+zx_0).$ This action is transitive  on $\R^2$ and transitive and free on the set   $M$ of incident pairs $(q,\ell)$, where $q\in\R^2$ and $\ell\subset \R^2$ is a non-vertical line  through $q$. There are $\mathrm H$-invariant line fields $L_1, L_2\subset TM$, where $L_1$ (resp. $L_2$) is tangent to the fibers of the projection $(q,\ell)\mapsto q$ (resp. $(q,\ell)\mapsto \ell$).

Let $q_0=(0,0), \ell_0=\{y=0\}$ (the real axis). Let $(X_1, X_2, X_3)$ the (left-invariant) frame on $\mathrm H$ dual to $(\theta^1, \theta^2, \theta^3)$.
Then the Lie algebras $(L_1)_\id, (L_2)_\id$ of the stabilizers of $q_0, \ell_0$ are spanned by
$X_1, X_2$ (resp.). Thus,
$$
D=L_1\oplus L_2=(\theta^3)^\perp,\quad
L_1=\{\theta^2,\theta^3\}^\perp,\quad
 L_2=\{\theta^1,\theta^3\}^\perp,$$
 with an adapted coframe
 $$\eta^1:=\theta^1,\quad \eta^2:=\theta^2,\quad \eta^3:=-\theta^3.$$

Solving the structure equations \eqref{eq:str3}--\eqref{eq:str4}, we get $\alpha=\theta^4,\ K=0,$ where $\theta^4=(\d s)/ s$ (the Maurer--Cartan form on $\R^*$), which gives, using equations \eqref{eq:feff2}--\eqref{eq:feff3},
$\sigma=-(2/3)\theta^4$ and
\be\label{eq:hfef22}
\grm=\theta^1 \cdot \theta^2+\frac{2}{3} \theta^3\cdot  \theta^4.
\ee

\begin{lemma}Null geodesics of \eqref{eq:hfef22}, projected to $\mathrm H$ and passing through $\id\in \mathrm H$ at $t=0$,  are of the form
$$ x=b(1-e^{-ct}),\  y=-ab(1-e^{-ct}), \  z=a(e^{ct}-1), \qquad a,b,c\in \R.
$$
They correspond to  chains $(q_t, \ell_t)\in M$, passing through $(q_0, \ell_0)$ at $t=0,$ where $q_t$ moves along the  line $\ell_*$ through   $q_0$  of  slope $-a$, and $\ell_t$ is a line through $q_t$ and $q_*=(b,0)$.
\end{lemma}
\begin{proof} The metric \eqref{eq:hfef22}  is a left-invariant metric on $G=\mathrm{H}\times\R^*,$ with an inertia operator $A:\g\to\g^*$
$$A={1\over 6}\left(
\begin{array}{cccc}
0& 3& 0 & 0 \\
 3 & 0 & 0 & 0 \\
0 & 0 & 0 & 2 \\
0& 0 & 2 & 0 \\
\end{array}
\right).
$$
The geodesic flow  on $T^*G$ projects via left translation to the Euler equations on $\g^*$,
$\dot \M=\mathrm{ad}_{A^{-1}\M}^*\M$, where $\mathrm{ad}_X^*=(\mathrm{ad}_X)^t\in\End(\g^*),$ $X\in\g$, $P\in\g^*$  and $\ad_XY=[X,Y].$ These are the Hamiltonian equations $\dot \M=\{H,\M\}$ with respect to  the standard Lie-Poisson structure on $\g^*$, where  $H={1\over 2}(\M,A^{-1}\M).$ See \cite[page 66]{AG}. Equivalently, $\dot X=A^{-1}\mathrm{ad}_X^*AX$.
To write these down explicitly with respect to  our bases, we first represent $X\in\g$ and $\mathrm{ad}_X^*\in\End(\g^*)$ by  $4\times 4$ matrices
$$X=\left(
\begin{array}{cccc}
0& x^1 & x^3 & 0 \\
0 & 0 & x^ 2& 0 \\
 0 & 0 & 0 & 0 \\
0& 0 & 0 & x^4 \\
\end{array}
\right), \qquad  \mathrm{ad}_X^*=(\mathrm{ad}_X)^t=
\left(
\begin{array}{cccc}
 0 & 0 & -x^2 & 0 \\
 0 & 0 &x^1 & 0 \\
 0 & 0 & 0 & 0 \\
 0 & 0 & 0 & 0 \\
\end{array}
\right),
$$
so $\dot X=A^{-1}\mathrm{ad}_X^*AX$ 
becomes
\be\label{eq:eulerH}
\dot x^1={2 \over 3}x^1 x^4,\quad
\dot x^2= -{2 \over 3}x^2 x^4,\quad
\dot x^3=\dot x^4=0.
\ee
 The general solution, with $H=(AX,X)/2=
x^1x^2/2+x^3 x^4/3=0$ (we are interested in the zero level set because we are computing the null geodesics), is
\be\label{eq:esol}
x^1=ae^{ct},\quad x^2=be^{-ct},\quad x^3=-{ab\over c},\quad x^4={3c\over 2},
\ee
where $a,b,c\in\R, $ $c\neq 0$. (In addition to these solutions there are some fixed points, which we now ignore).

Now let $g(t)\in H\times\R^*$ be a null geodesic, with
$$g(t)=  \left(\begin{array}{cccc}1&z&y&0\\  0&1&x&0\\
0&0&1&0\\ 0&0&0&s\end{array}\right),
\quad x,y,z,s\in\R, \ s\neq 0.$$
Then $X=g^{-1}\dot g\in\g$ is given by \eqref{eq:esol}.
Explicitly,
\be
\dot x=x^2=b\,e^{-ct}, \quad
\dot y-z\dot x=x^3=-{ab\over c}, \quad
\dot z=x^1=a\,e^{ct}
\ee
(we do not  need the  $s$ equation). Change the time variable to $\tau=ct$,  denoting derivative with respect to  $\tau$ by $(\ )'$ and renaming  the constants, $a\mapsto a/c, b\mapsto b/c$, we get
\be\label{eq:sys}x'=be^{-\tau}, \quad y'-zx'=-ab,\quad  z'=ae^{\tau}.
\ee
Consider chains through $\id\in \mathrm H$,  i.e. $x_0=y_0=z_0=0.$
Then $z=a(e^{\tau}-1),$ hence $y'=zx'-ab=-ab\,e^{-\tau}.$    The solution of \eqref{eq:sys} is then
$$x=b(1-e^{-\tau}),\quad y=-ab(1-e^{-\tau}), \quad z=a(e^{\tau}-1).
$$
Thus $(x,y)$ traces a line of slope $-a$ through the origin, and each  line of slope $z$ through $(x,y)$ passes through $(b,0).$
\end{proof}

\subsubsection{Second proof of Proposition \ref{prop:flat}: via  the Euclidean group}\label{sect:euc}

Here $M= \P(T\R^2)=\R^2\times \P(\R^2)$ is the set of pairs $(q,\ell)$ with $\ell$ a line through $q$,  $L_1\subset TM$ is tangent to the fibers of the projection onto the first factor and similarly for $L_2$. The group  $\SE$ of orientation-preserving isometries of $\R^2$ acts transitively on $M$, with stabilizer $\Z_2$ (reflection about a point), preserving $L_1, L_2$. Fixing a point $(q_0, \ell_0)\in M$ identifies $M$ with $\SE/\Z_2$, and hence equips $\SE$ with left-invariant line fields $L_1, L_2$ given by a pair of 1-dimensional subspaces $(L_1)_\id, (L_2)_\id$, the Lie algebras of the stabilizers of $q_0, \ell_0$ (resp.).

Identify  $\R^2=\C$ with the affine plane $\z=1$  in $\C^2$, $\z\mapsto (\z,1)$; then $\SE$ is identified with the subgroup of $\GL_2(\C)$ consisting of matrices of the   form
\be\label{eq:E} \left(\begin{array}{ccc}e^{i\theta}&\z\\  0&1\end{array}\right),
\quad \z\in\C,\ \theta\in\R.
\ee
Its Lie algebra  $\se$  consists of matrices of the form
\be\label{eq:Ee}\left(\begin{array}{cc}
ix^1&x^2+ix^3\\ 0&0\end{array}\right), \quad x^i\in\R.
\ee
Let $\theta^j$ be the  left-invariant 1-form on $\SE$ whose value at $\id$ is
$x^j$, $j=1,2,3$.  Then
$$\Theta:=\left(\begin{array}{cc}i \theta^1& \theta^2+i \theta^3\\ 0&0\end{array}\right)$$
is the left-invariant Maurer--Cartan form on $\SE$, satisfying $\d\Theta=-\Theta\wedge\Theta$,  from which we get
\be\label{eq:stre}
\d  \theta^1=0, \ \d \theta^2=\theta^1\wedge \theta^3, \   \d \theta^3=-\theta^1\wedge \theta^2.
\ee
Let $X_1, X_2, X_3$ the left-invariant vector fields on $\SE$ dual to $\theta^1, \theta^2, \theta^3$. Let $q_0=0, \ell_0=\R$ (the real axis).
Then the Lie algebras of the stabilizers of $q_0, \ell_0$ are spanned by
$X_1, X_2$ (resp.). Thus
$$
D=L_1\oplus L_2=(\theta^3)^\perp,\quad
L_1=\{\theta^2,\theta^3\}^\perp,\quad
L_2=\{\theta^1,\theta^3\}^\perp,$$
 with an adapted coframe
 $$\eta^1:=\theta^1,\quad \eta^2:=\theta^2,\quad \eta^3:=-\theta^3.$$

Solving the structure equations \eqref{eq:str3}--\eqref{eq:str4}, we get $\alpha=\theta^4,\ K=0,$ where $\theta^4=(\d s)/ s$ (the MC form on $\R^*$), which gives, using equations \eqref{eq:feff2}--\eqref{eq:feff3},
$\sigma=-(2/3)\theta^4$ and
\be\label{eq:fef2}\grm=\theta^1 \cdot \theta^2+\frac{2}{3} \theta^3 \cdot \theta^4.
\ee
This is a left-invariant metric on $G=\SE\times\R^*,$ with an inertia operator $A:\g\to\g^*$
$$A={1\over 6}\left(
\begin{array}{cccc}
0& 3& 0 & 0 \\
 3 & 0 & 0 & 0 \\
0 & 0 & 0 & 2 \\
0& 0 & 2 & 0 \\
\end{array}
\right).
$$
The geodesic flow  on $T^*G$ projects via left translation to the Euler equations on $\g^*$,
$\dot \M=\mathrm{ad}_{A^{-1}\M}^*\M$, where $\mathrm{ad}_X^*=-(\mathrm{ad}_X)^t\in\End(\g^*).$ These are the Hamiltonian equations $\dot \M=\{H,\M\}$ with respect to  the standard Lie-Poisson structure on $\g^*$, where  $H={1\over 2}(\M,A^{-1}\M).$ See \cite[p. 66]{AG}.
To write these down explicitly with respect to  our bases, we first represent $X\in\g$ and $\mathrm{ad}_X^*\in\End(\g^*)$ by  $4\times 4$ matrices
$$X=\left(
\begin{array}{cccc}
0& -x^1 & x^2 & 0 \\
x^1 & 0 & x^ 3& 0 \\
 0 & 0 & 0 & 0 \\
0& 0 & 0 & x^4 \\
\end{array}
\right), \qquad  \mathrm{ad}_X^*=(\mathrm{ad}_X)^t=
\left(
\begin{array}{cccc}
 0 & x^3 & -x^2 & 0 \\
 0 & 0 & x^1 & 0 \\
 0 & -x^1 & 0 & 0 \\
 0 & 0 & 0 & 0 \\
\end{array}
\right),
$$
so $\dot \M=\mathrm{ad}_{A^{-1}\M}^*\M$ 
becomes
\be\label{eq:euler}
\dot \M_1=-2 \M_1 \M_3 +3 \M_2 \M_4, \quad
\dot \M_2= 2 \M_2 \M_3,\quad
\dot \M_3= -2 \M_2^2,\quad
\dot \M_4=0.
\ee
with constants of motion (in addition to $\M_4$),
\begin{align*}
H&={1\over 2}(\M,A^{-1}\M)=
2\M_1 \M_2+3 \M_3 \M_4=0
,\qquad
k=(\M_2)^2+(\M_3)^2.
\end{align*}
Let us use polar coordinates in the $\M_2\M_3$-plane:
$$\M_2=r\cos\phi,\quad \M_3=r\sin \phi.$$
Then $$\dot\phi=-2r\cos\phi,\quad \M_1=c\tan\phi, \quad \M_4=-2c/3, \quad c=const., \quad r=const.  $$
Now let $g(t)\in \SE\times\R^*$ be a null geodesic, with
$$g(t)=  \left(\begin{array}{ccc}e^{i\theta}&\z&0\\  0&1&0\\
0&0&s\end{array}\right),
\quad \z\in\C,\ \theta,s\in\R^*.$$
Let $X=g^{-1}\dot g\in\g$. Then $\M=AX$ satisfies Euler equations \eqref{eq:euler}. Explicitly,
$$\dot \theta= x^1=2\M_2=-\dot\phi,\quad
\dot \z= e^{i\theta} (x^2+ix^3)=e^{i\theta}(2\M_1+i3\M_4)=2ce^{i\theta}(\tan\phi-i).$$
(The $s$ equation is omitted; it will not be used). Assume, without loss of generality, that $g(0)=\id,$ i.e. $\theta(0)=0$ and $\z(0)=0$, so
$\theta=\phi_0-\phi$. We reparametrize $g(t)$ by $\phi$, denote derivative with respect to  $\phi$ by $(\ )'$, and get  $\z'=i(c/r)e^{i\phi_0}\dot \z/\dot\phi=i(c/r)e^{i\phi_0}\sec^2\phi.$ Integrating yields $\z=i(c/r)e^{i\phi_0}\tan\phi.$ Now we rotate the chain by $-\phi_0$, reflect about the $x$-axis and rename $c$, so that
\be\label{eq:flat_chains}
\z=ic\tan\phi, \quad \theta=\phi, \quad c\in\R.
  \ee
This corresponds to a chain $(q_\phi,\ell_\phi)$, where $q_\phi$  moves along $\ell^*=$the $y$ axis and $\ell_\phi$ is the line connecting $q_*=(-c,0)$ with $q_\phi$.

  \subsection{Circles of fixed radius }

Here $M\subset \C\times\C$ is the set of pairs of points $(p,q)$ with $|p-q|=1$,  $L_1, L_2\subset TM$ are tangent to the fibers of the projection onto the first (resp. second) factor. The first projection maps the fibers of the second projection to the set of plane circles of radius 1. The group  $\SE$ of orientation preserving isometries of $\C=\R^2$ acts transitively and freely on $M$, preserving $L_1, L_2$. We use the same notation for this group as in Section \ref{sect:euc}. Let $p_0=0, q_0=1.$
Then the Lie algebras $(L_1)_\id, (L_2)_\id$ of the stabilizers of these points are spanned by
$X_1, X_1-X_3$ (resp.). Thus
$$
D=L_1\oplus L_2=(\theta^2)^\perp,\quad
L_1=\{\theta^2,\theta^3\}^\perp,\quad
L_2=\{\theta^2,\theta^1+\theta^3\}^\perp,$$
 with an adapted coframe
 $$\eta^1:=\theta^3,\quad \eta^2:=\theta^1+\theta^3,\quad \eta^3:=-\theta^2.$$

Solving the structure equations \eqref{eq:str3}--\eqref{eq:str4}, we get $\alpha=\theta^2+\theta^4,\ K=-1,$ where $\theta^4=(\d s)/ s$ (the Maurer--Cartan  form on $\R^*$), which gives, using equations \eqref{eq:feff2}--\eqref{eq:feff3},
$\sigma=-\theta^2/2-2\theta^4/3$ and
\be\label{eq:fef2}\grm={1\over 2}(\theta^2)^2+(\theta^3)^2+\theta^1\cdot  \theta^3+\frac{2}{3} \theta^2 \cdot \theta^4.
\ee
This is a left-invariant metric on $G=\SE\times\R^*,$ with an inertia operator $A:\g\to\g^*$
$$A={1\over 6}\left(
\begin{array}{cccc}
0& 0 & 3 & 0 \\
 0 & 3 & 0 & 2 \\
 3 & 0 & 6 & 0 \\
0& 2 & 0 & 0 \\
\end{array}
\right).
$$
The geodesic flow  on $T^*G$ projects via left translation to the Euler equations on $\g^*$,
$\dot \M=\mathrm{ad}_{A^{-1}\M}^*\M$.
To write these down explicitly with respect to  our bases, we first represent $X\in\g$ and $\mathrm{ad}_X^*\in\End(\g^*)$ by  $4\times 4$ matrices
$$X=\left(
\begin{array}{cccc}
0& -x^1 & x^2 & 0 \\
x^1 & 0 & x^ 3& 0 \\
 0 & 0 & 0 & 0 \\
0& 0 & 0 & x^4 \\
\end{array}
\right), \qquad  \mathrm{ad}_X^*=(\mathrm{ad}_X)^t=
\left(
\begin{array}{cccc}
 0 & x^3 & -x^2 & 0 \\
 0 & 0 & x^1 & 0 \\
 0 & -x^1 & 0 & 0 \\
 0 & 0 & 0 & 0 \\
\end{array}
\right),
$$
so $\dot \M=\mathrm{ad}_{A^{-1}\M}^*\M$  
 become
\begin{align}\label{eq:eulerc}
\begin{split}
\dot \M_1&= 2 \M_1 \M_2 - 3 \M_3 \M_4,\quad
\dot \M_2=  2 \M_3(\M_3-2\M_1) ,\\
\dot \M_3&= -2\M_2 (\M_3-2\M_1),\quad
\dot \M_4=0,
\end{split}
\end{align}
with constants of motion (in addition to $\M_4$),
\begin{align*}
H   &= {1\over 2}(\M,A^{-1}\M)
     = -2 \M_1^2 + 2 \M_1 \M_3 + 3 \M_2 \M_4 - \frac{9}{4} M_4^2 \\
r^2 &= (\M_2)^2+(\M_3)^2.
\end{align*}
We make the following change of variables:
\be\label{eq:cv}
y=4 \M_1- 2 \M_3,\quad \M_2=r\cos\phi,\quad \M_3=-r\sin \phi, \quad \M_4=c/3.
\ee
Then \eqref{eq:eulerc} reduces to
\be\label{eq:y1}
\dot\phi=-y, \quad \dot y=4 r  (c-r \cos\phi )\sin \phi,\qquad c,r\in \R, \ r\geq 0,
\ee
and the nullity condition $H=0$ becomes
\be\label{eq:y}
y^2=8c r \cos\phi+4 r^2 \sin ^2\phi-2c^2.
\ee
\begin{remark}
Equations \eqref{eq:y1} can be written as a single Newton type second order ODE, $\ddot\phi=f(\phi),$
where $f(\phi)=4 r  (r \cos\phi -c)\sin \phi.$ As usual, one writes $f(\phi)=-U'(\phi)$, say
$U(\phi)=-4cr\cos\phi-2r^2\sin^2\phi.$ Then $\dot\phi^2/2+U(\phi)$ is constant along solutions of $\ddot\phi=f(\phi)$ (`conservation of energy'). Equation \eqref{eq:y} fixes the value of this energy.
\end{remark}

Now let $g(t)\in \SE\times\R^*$ be a null geodesic, with
$$g(t)=  \left(\begin{array}{ccc}e^{i\theta}&\z&0\\  0&1&0\\
0&0&s\end{array}\right), \quad \z\in\C,\ \theta,s\in\R^*.$$
Let  $X=g^{-1}\dot g\in\g$. Then $\M=AX$ satisfies \eqref{eq:eulerc}. Explicitly,
$$
\dot \z= e^{i\theta} (x^2+ix^3)=e^{i\theta} (3 \M_4+i2 \M_1),\quad
\dot \theta= x^1=-4 \M_1 + 2 \M_3.
$$
Using the change of variables \eqref{eq:cv}, we get
$$
\dot \z= e^{i\theta}[ c-i(\dot\phi/2+r \sin \phi)],\
\dot \theta= \dot\phi,$$
where $\phi(t)$ satisfies equations \eqref{eq:y1}--\eqref{eq:y}. For a fixed $\phi(t)$ these equations  are invariant under rigid motions   (adding a constant angle to $\theta$, rotating $\z$ by this angle and translating $\z$ by some constant vector). So we can assume, without loss of generality,  that   $\theta=\phi$. Hence
$$
\dot \z= e^{i\theta}[ c-i(\dot \theta/2+r \sin \theta)], \quad
\dot\theta^2=8c r \cos\theta+4 r^2 \sin ^2\theta-2c^2.
$$
Next we use the   scaling invariance, $t\mapsto\lambda t,
\ c\mapsto \lambda c,\ r\mapsto \lambda r,$ to assume $r=1$. We can also use the reflection symmetry $t\mapsto -t, \theta\mapsto\theta+\pi, \z\mapsto-\z,c\mapsto -c$ to assume that $c\geq 0$. Thus every chain, up to a rigid motion and reparametrization, is a solution to
\begin{align}
\dot \z&= e^{i\theta}[ c-i(\dot \theta/2+\sin \theta)], \label{eq:cc}\\
(\dot\theta)^2&=8c\cos\theta+4\sin ^2\theta-2c^2, \label{eq:elip}
\end{align}
with $c\in\R,\ c\geq 0.$
\begin{lemma} Let $F(\theta, c)=8c\cos\theta+4\sin ^2\theta-2c^2$ (the right-hand side of equation \eqref{eq:elip}). Then $F\geq 0$ has a solution if and only if   $|c|\leq 4.$ For every  $c\in[0,4]$ the set  of $\theta\in [-\pi,\pi]$ such that $F(\theta, c)\geq 0$ is an  interval $[-\theta_{max}, \theta_{max}]$,  where $\theta_{max}=\cos ^{-1}\left(c-\sqrt{c^2/2+1}\right)\in [0,\pi]$. For $c\in(0,4)$ every solution $\theta(t)$ of \eqref{eq:elip} oscillates between $-\theta_{max}$ and $\theta_{max}$. If $c=0$ then  $\lim\theta$ is $0$ or $\pi$ as $t\to\pm \infty$. If $c=4$ then $\theta\equiv 0$.
\end{lemma}

\begin{proof}We write $F= -4x^2+8cx + 4- 2c^2,$ where  $x=\cos\theta$. The roots of this polynomial  are $x_\pm=c\pm\sqrt{1+c^2/2}$ and $F>0$ in the interval $(x_-, x_+)$. To be able to solve for $\theta$ we need $[x_-, x_+]$ to intersect the interval  $[-1,1]$. It is elementary to show that this occurs if and only if   $|c|\leq 4.$
\end{proof}

\begin{figure}
\centering
\def\svgwidth{1\textwidth}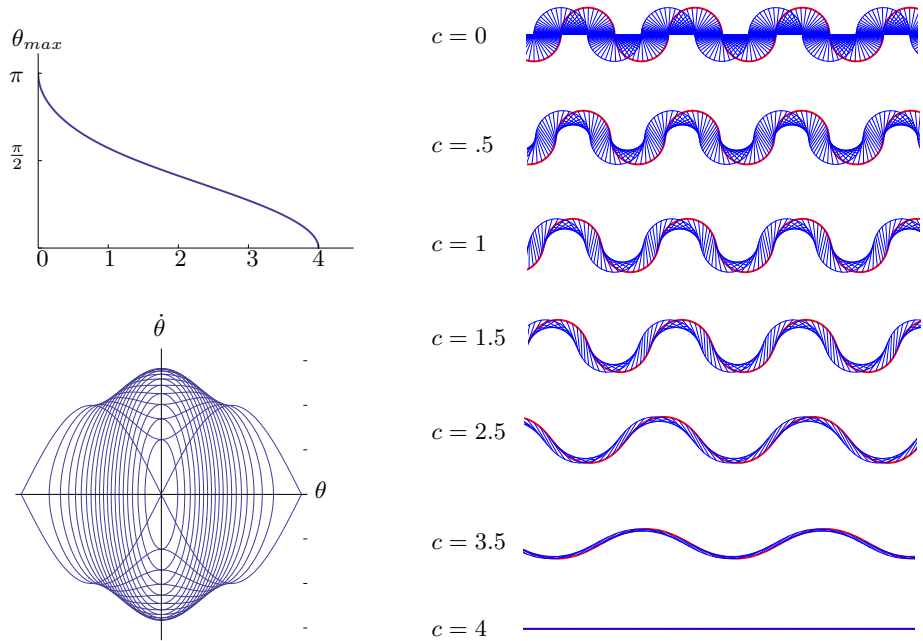
\caption{Chains of circles path geometry (solutions of equations \eqref{eq:cc}--\eqref{eq:elip}).
 Top left: plot of the maximum amplitude of oscillation of $\theta$ as a function of the chain parameter $c\in[0,4].$ Bottom left: phase curves of equation  \eqref{eq:elip} for various  $c$ values. Right: each red curve is the projection of the chain on the Euclidean plane. The blue curve represents the projection of the chain on the dual plane; it is formed by joining the tips of the unit vectors  in the direction $\theta$ at each point $\z$ of the red curve (the thin blue lines). }\label{fig:circle_chains}
\end{figure}

\begin{proposition}
Every chain of the path geometry of circles of radius 1 in the Euclidean plane, up to an affine reparametrization and rigid motion, is given by a unique solution of equations \eqref{eq:cc}--\eqref{eq:elip} with $c\in[0,4)$, $\z(0)=\theta(0)=0$ (for $c=0$ one should take $\theta(0)\neq 0,\pi.$
\end{proposition}
See Figure \ref{fig:circle_chains}. The  projection of the chains on the Euclidean plane (the curves $\z(t)$) look like inflectional elastica, but they are not (checked numerically).

\subsection*{Further properties/questions  about these chains:}
\begin{enumerate}
\item From the pictures, $\z(t)+e^{i\theta(t)}$ (the red curve) is obtained from $\z(t)$ by translation and parameter shift. Presumably, this comes out of equations \eqref{eq:cc}--\eqref{eq:elip}. Is this a manifestation of the self-duality of this path geometry? How exactly?
\item One should be able to write explicit solutions of equations \eqref{eq:cc}--\eqref{eq:elip} using elliptic functions. See \cite{P}.
%

\mn {\em Note.} One can write down an explicit general solution for the case $c = 0$ without any special functions, and one can verify analytically that the arcs are semicircles. So, as embedded submanifolds they are $C^1$ but not $C^2$ at inflection points. In particular, the chain ODE is not satisfied at these inflection points.
\item Equation \eqref{eq:elip} is the equation of a pendulum under a strange  force law:
$f(\theta)=4 (\cos\theta -c)\sin \theta,$ with special initial conditions: $\theta(0)=0, \dot\theta(0)=8c-c^2$ (for $c=0$ it is the homoclinic solution of the pendulum equation $\ddot\theta=2\sin2\theta$. Is there a good mechanical/geometrical interpretation  of  this motion?

\item The chains of this geometry project to a 1-parameter family of curves in $\R^2$ (up to rigid motion). Is there a simple geometric description of this family? Our first guess was elastica but it is not the case.

\item In the pictures, there are points along $\z(t)$ at which $\theta(t)$ is the direction of the tangent $\dot\z(t)$ (the inflection points of the red curves on the right of Figure \ref{fig:circle_chains}). Is this phenomenon unavoidable?
%

\end{enumerate}

\subsection{Hooke ellipses of fixed area}
The manifold
\[
    M=\{(x,y,E,F,G)\in\R^5\st Ex^2+2Fxy+Gz^2=1, \ EG-F^2=1, \ E>0\}
\]
 parametrizes the set of incident pairs $(\r,\cE)$, where $\r=(x,y)^t\in\R^2\setminus \{0\}$ and
 $\cE$ is an ellipse centered at the origin (a `Hooke ellipse') of area $\pi$.
 \begin{proposition}\label{prop:ppg}
  The path geometry in $\R^2\setminus \{0\}$ of Hooke ellipses of fixed area is projective (the paths are the unparametrized geodesics of a torsion-free affine connection).
 \end{proposition}
 \begin{proof}
As mentioned before, this  is equivalent to showing  that the associated  ODE $y''=f(x,y,y')$ is cubic in $y'$.  Let $\H=\{(E,F,G)\st EG-F^2=1, E>0\}$ be the path space. We parametrize $\H$ by  the upper half-plane  $\R^2_+=\{(a,b)\st b>0\}$,
\be\label{eq:uhp}
(a,b)\mapsto{1\over b}(1, -a ,a^2+b^2).
\ee
  Hooke ellipses of area $\pi$ are then given by equations of the form
\be\label{eq:huh}
 x^2 -2 a x y+ (a^2+b^2 )y^2=b, \quad a,b\in\R, b>0.
\ee
 Assuming $y=y(x)$ in this equation and taking two derivatives with respect to $x$, we get
 \begin{align*}
&x -a (y+x y')+(a^2+b^2)yy'=0,\\
&1-a(2y'+xy'')+ (a^2+b^2 )\left[(y')^2+y y''\right]=0.
\end{align*}
Eliminating $a,b$ from the last 3 equations  and solving for $y''$, we obtain
$$y''=(xy'-y)^3.$$

Another  proof, more direct, consists of showing that Hooke  ellipses of area $\pi$ are the (unparametrized) geodesics of a Riemannian metric in $\R^2\setminus \{0\}$, given in polar coordinates by
$ds^2={dr^2}/\Delta^2+{r^2d\theta^2}/\Delta,$ $\Delta=1+cr^2+r^4,$  $c\in\R$.

See \cite{BJ} for yet another proof, via  equivalence with the path geometry of Kepler ellipses of fixed major axis, which is projective since these are geodesics of the Jacobi-Maupertuis metric of the Kepler problem.
 \end{proof}

 \paragraph{Fefferman metric.}
 Let $L_1\subset TM$ be the tangents to the fibers of the projection on the first component,
$(q,\cE)\mapsto q$, and similarly for $L_2$. The group $\SLt$ acts transitively and freely on $M$ via its standard linear action on $\R^2$, preserving $L_1, L_2$. Fixing a point  $(q_0,\cE_0)\in M$ identifies $M$ with $\SLt$, and $L_1, L_2$ with two  left-invariant line fields on $\SLt$, given at $\id\in\SLt$ by the Lie algebras of the stabilizers  of $q_0,\cE_0$, respectively.

The Lie algebra $\slt$ of $\SLt$ consists of matrices of the form
 $$\left(\begin{array}{rr}
x^1 &x^2\\
x^3 & -x^1
\end{array}\right),\quad x^i\in\R.
$$
The left-invariant $\slt$-valued Maurer--Cartan form  on $\SLt$ is
\be\label{eqn:mcsl2}\Theta=g^{-1}\d g= \left(\begin{array}{rr}
\theta^1 &\theta^2\\
\theta^3 & -\theta^1
\end{array}\right).
\ee
 The Maurer--Cartan equation $\d\Theta=-\Theta\wedge\Theta$ gives
 \be\label{eq:mc2}
 \d\theta^1=-\theta^2\wedge\theta^3, \  \d\theta^2=-2\theta^1\wedge\theta^2, \  \d\theta^3=2\theta^1\wedge\theta^3.
\ee
Fix
$q_0:=(1,0)^t$,  $\cE_0:=\{x^2+y^2=1\}$. Then
$$(L_1)_\id=\Span\left(\begin{array}{cc}
0&1\\
0&0
\end{array}\right),\ (L_2)_\id=\Span\left(\begin{array}{cc}
0&-1\\
1&0
\end{array}\right),\ D=L_1\oplus L_2=\Ker(\theta^1).$$
An adapted coframe is thus

$$\eta^1=\theta^2+\theta^3,\quad
 \eta^2=\theta^3, \quad \eta^3:=-\theta^1.
$$
We use this coframe to trivialize the associated $\R^*$-structure $B\simeq \SLt\times \R^*$ and put the standard coordinate $s$ on the $\R^*$ factor. The associated 1-forms on $B$ are
 $$\omega^1={1\over s}(\theta^2+\theta^3),\quad
\omega^2=s\theta^3,\quad
\omega^3=-\theta^1.
$$
Solving the structure equations \eqref{eq:str3}-\eqref{eq:str4}, we get $\alpha=2\theta^1+\theta^4,\ a_1=4/s^2,\  a_2=0, \ K=-2,$ where $\theta^4=(\d s)/s$ (the MC form on $\R^*$), which gives, using equations \eqref{eq:feff2}--\eqref{eq:feff3},
$\sigma=-\theta^1-(2/3)\theta^4$ and
\be\label{eq:fef5}
\grm=(\theta^1)^2+(\theta^3)^2+\theta^2 \theta^3+\frac{2}{3} \theta^1 \theta^4.
\ee

\bn{\bf Hooke chains (null geodesics of the Fefferman metric).}
The pseudo-Riemannian metric \eqref{eq:fef5} is a left-invariant metric on the Lie group $G:=\SLt\times \R^*.$ Let $\g=\slt\times\R$ be its Lie algebra and  $A:\g\to\g^*$ the `inertia' operator corresponding to the quadratic form \eqref{eq:fef5}; that is, $\grm(X,Y)=(AX)Y,$ $X,Y\in\g$. Then
$$A={1\over 6}\left(
\begin{array}{cccc}
6 & 0 & 0 & 2 \\
 0 & 0 & 3 & 0 \\
 0 & 3 & 6 & 0 \\
 2& 0 & 0 & 0 \\
\end{array}
\right)
$$
(with respect to  the basis $\{\theta^i\}$ and its dual).  As in previous examples, the geodesic flow  on $T^*G$ projects  to  $\dot \M=\mathrm{ad}_{A^{-1}\M}^*\M$ on  $\g^*$,
  the Hamiltonian equations with respect to  the standard Lie-Poisson structure on $\g^*$   with Hamiltonian   $H={1\over 2}(\M,A^{-1}\M).$
To write these down explicitly, we first represent $X\in\g$ and $\mathrm{ad}_X^*\in\End(\g^*)$ by   the matrices
$$X=\left(
\begin{array}{ccc}
x^1&  x^2 & 0 \\
x^3 &-x^1 &  0 \\
0 & 0 & x^4
\end{array}
\right), \qquad  \mathrm{ad}_X^*=
\left(
\begin{array}{cccc}
 0 & -2 x^2 & 2 x^3& 0 \\
 -x^3 & 2 x^1 & 0 & 0 \\
 x^2 & 0 & -2 x^1 & 0 \\
 0 & 0 & 0 & 0 \\
\end{array}
\right)
,
$$
so $\dot \M=\ad_{A^{-1}\M}^*\M$  becomes
\begin{eqnarray}\label{eq:euler1}
\begin{split}
\dot \M_1&=8 (\M_2)^2,\quad
\dot \M_2=2 \M_2(3 \M_4-\M_1), \\
\dot \M_3&=2\M_1(\M_3-2\M_2)-6\M_3\M_4, \quad
\dot \M_4=0,
\end{split}
\end{eqnarray}
with constants of motion $\M_4,k, H$, where
\be\label{eq:euler2}
 k=(\M_1)^2+4\M_2\M_3, \quad
 H=2 \M_2\M_3-2 (\M_2)^2+3 \M_1\M_4-9 \M_4^2/2=0.
\ee
Note that $k$ is a Casimir of $\g^*$ coming  from the Killing form of $\slt$. We set $H=0$ since we are looking for null geodesics.
Next we make the following change of variables
$$\M_1=b(c+\sin\phi),\quad \M_2={b\over 2}\cos\phi,\quad \M_3
=b\left(\cos\phi -{p\over 2}\right),\quad \M_4={b c\over 3}.
$$
We have  $k-2H=b^2$, hence $b$ is constant. Since $P_4=bc/2$ is constant  $c$ is constant as well. Equations \eqref{eq:euler1}-\eqref{eq:euler2} then  reduce to
\be\label{eq:hchain}
\dot\phi=2b\cos\phi,\quad
p=\cos  \phi + c (c +2  \sin  \phi )\sec \phi.
\ee

 Next let $g(t)\in \SLt\times\R^*$ be a null geodesic, with
$$g(t)=  \left(\begin{array}{ccc}
x&z&0\\
y&w&0\\
  0&0&s
  \end{array}\right), \quad x,y,z,w,s\in\R,\ s\neq 0, \ xw-yz=1.$$
Let  $X=g^{-1}\dot g\in\g$. Then $\M=AX$ satisfies equations \eqref{eq:euler1}. Explicitly,
 \begin{align*}
&\dot x  =  x^1 x+ x^3 z=b[c x+(\cos\phi) z],
\ \quad\dot z  = x^2 x - x^1 z=-b[p\, x+cz], \\
&\dot y  = x^1 y+ x^3 w=b[cy+(\cos\phi)w],
\quad\dot w =  x^2 y- x^1 w=-b[p\,y+cw],
\end{align*}
where $p, \phi$ are  given by equation \eqref{eq:hchain}.
Denote $\r:=(x,y), \ \hb:=(z,w)\in\R^2$, then the last system is
\be\label{eq:sl2chains}
\dot\r=b\left[c\r+(\cos\phi)\hb\right],\quad \dot\hb=-b\left[p\r+c\hb\right].
\ee
\begin{lemma}\label{lemma:ca}
$\phi$ is  {\em twice the centro-affine arclength} of the projection of the  chain to the Hooke plane (the $\r$ plane).
\end{lemma}
\begin{proof} $\r,\hb$ are the columns of a matrix in $\SLt$, hence $[\r, \hb]=1$. It then follows from equations \eqref{eq:sl2chains}
that $[\r, \d \r/\d\phi]=[\r,\dot\r/\dot\phi]=[\r,\hb/2]=1/2.$
\end{proof}

Let us reparametrize the chains by  $\tau:=\phi/2$ (the centro-affine arclength) and denote derivative with respect to $\tau$ by $(\ )^\prime$. Equations \eqref{eq:sl2chains} now become
\begin{align}\label{eq:sl2chains1}
\begin{split}
&\r'=c(\sec2\tau)\r+\hb, \\
&\hb'=-\left[1+c(c+2(\sin2\tau))\sec^22\tau\right]\r-c(\sec2\tau)\hb.
\end{split}
\end{align}

\begin{lemma}\label{lemma:proj}
$$
\r''=- \r .$$
 \end{lemma}
\begin{proof}Straightforward calculation from equations \eqref{eq:sl2chains1}.
\end{proof}
 Thus, combined with $[\r,\r']=1$ (Lemma \ref{lemma:ca}), each Hooke chain projects to a Hooke ellipse of area $\pi$ in the $\r$ plane, as expected from Proposition \ref{prop:ppg} and Theorem \ref{thm:proj_chains}.

 \begin{proposition}\label{prop:hc}
 Every chain in  $\SLt$  of the path geometry of Hooke ellipses of area $\pi$, up to left translation,  is of the form
 $$\r=e^{i\tau}, \quad \hb=e^{i\tau}(-c\sec(2\tau)+i)$$
 (using complex notation),
 for some $c\in\R$, $c\neq 0$. See Figure \ref{fig:hooke_chains}.
 \end{proposition}

 \begin{proof}
 $\SLt$ acts transitively on Hooke ellipses of area $\pi$, hence the projection of the chain to the $\r$ plane can be brought to the unit circle. Parametrized by centro affine arc length, it is $\r=e^{i\tau}$. Then the 1st equation of \eqref{eq:sl2chains1} implies the formula for $\hb(\tau).$ For $c=0$ this formula produces a curve tangent to the contact distribution, which is excluded.
 \end{proof}

\begin{figure}
\centering
\includegraphics[width=.3\textwidth]{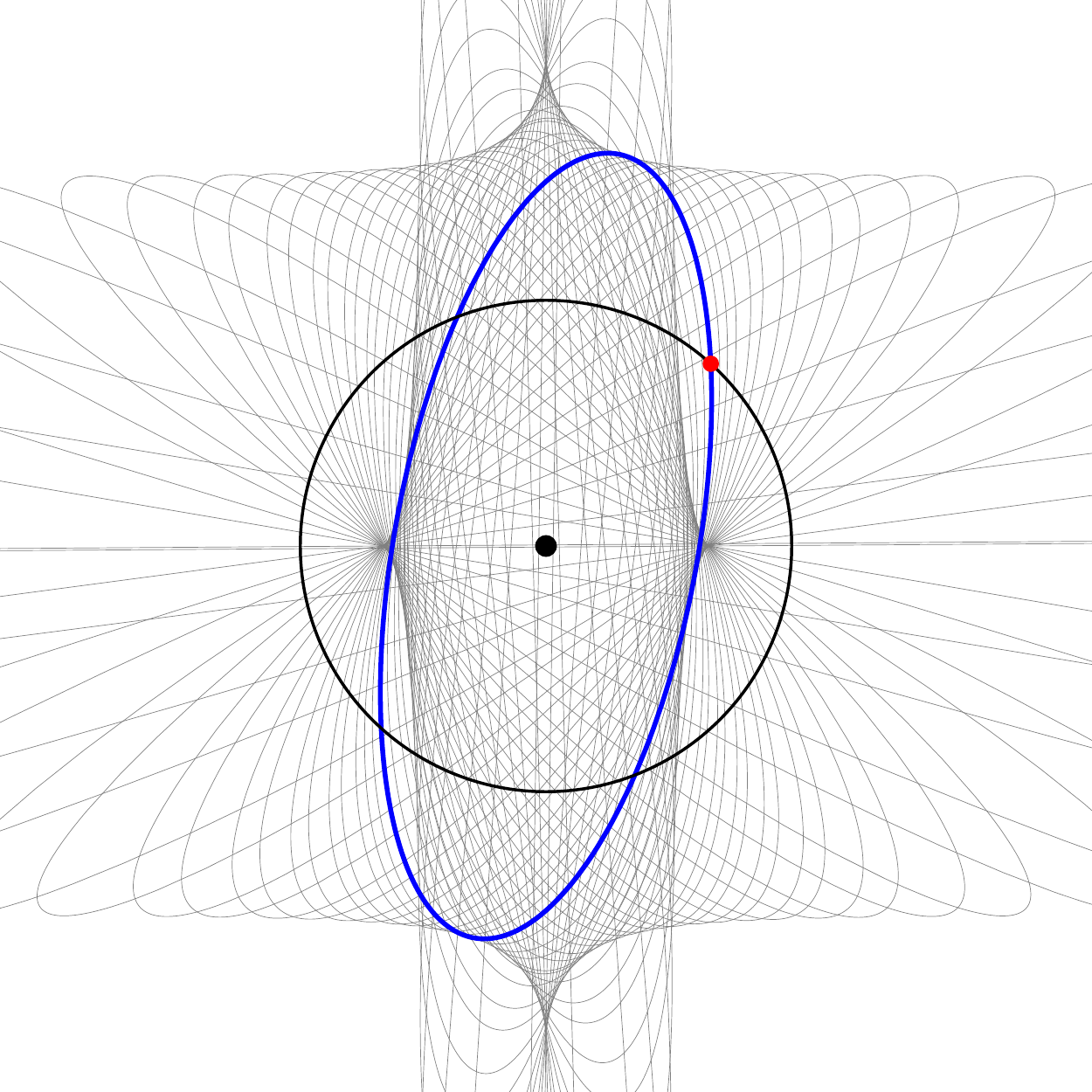}\quad
\includegraphics[width=.3\textwidth]{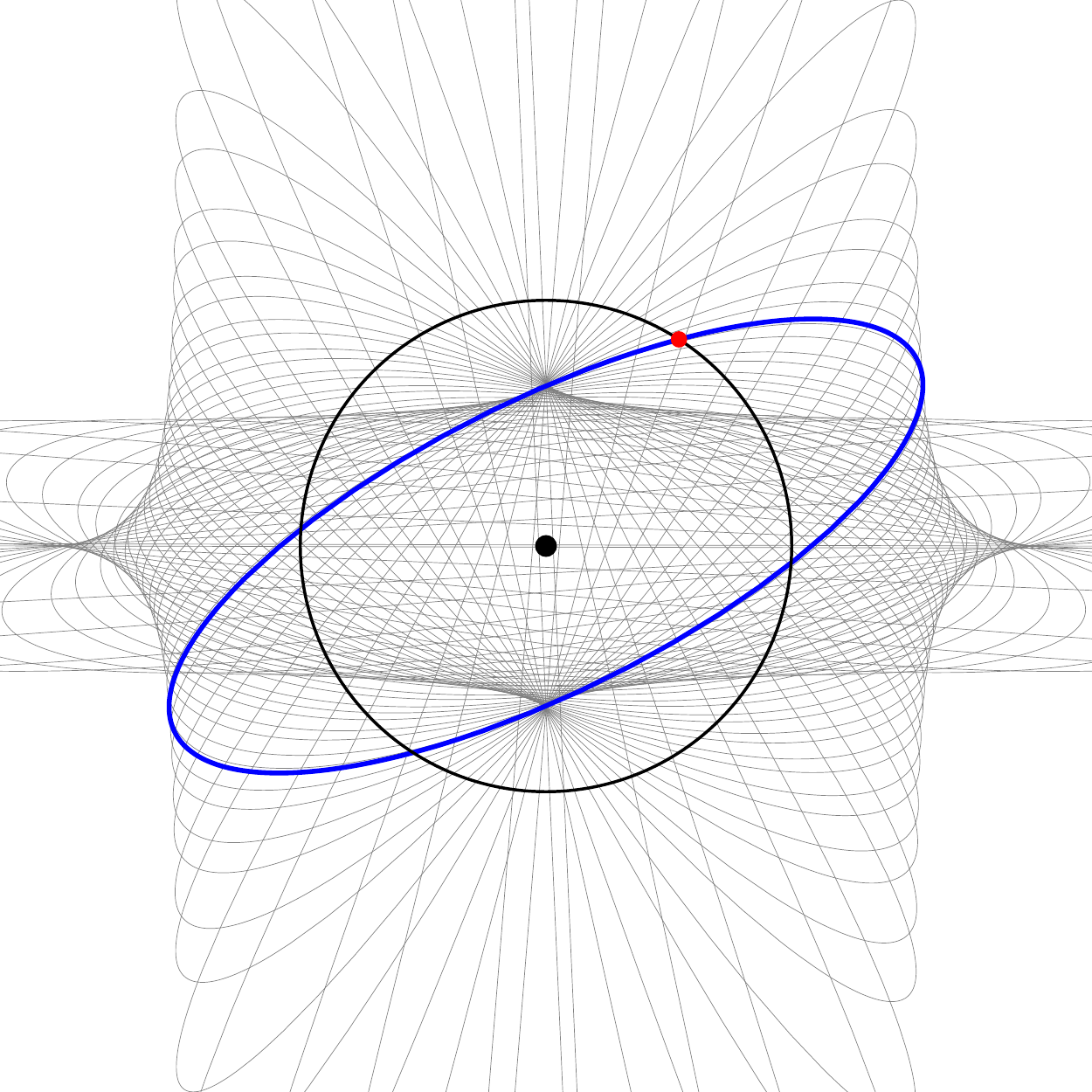}\quad
\includegraphics[width=.3\textwidth]{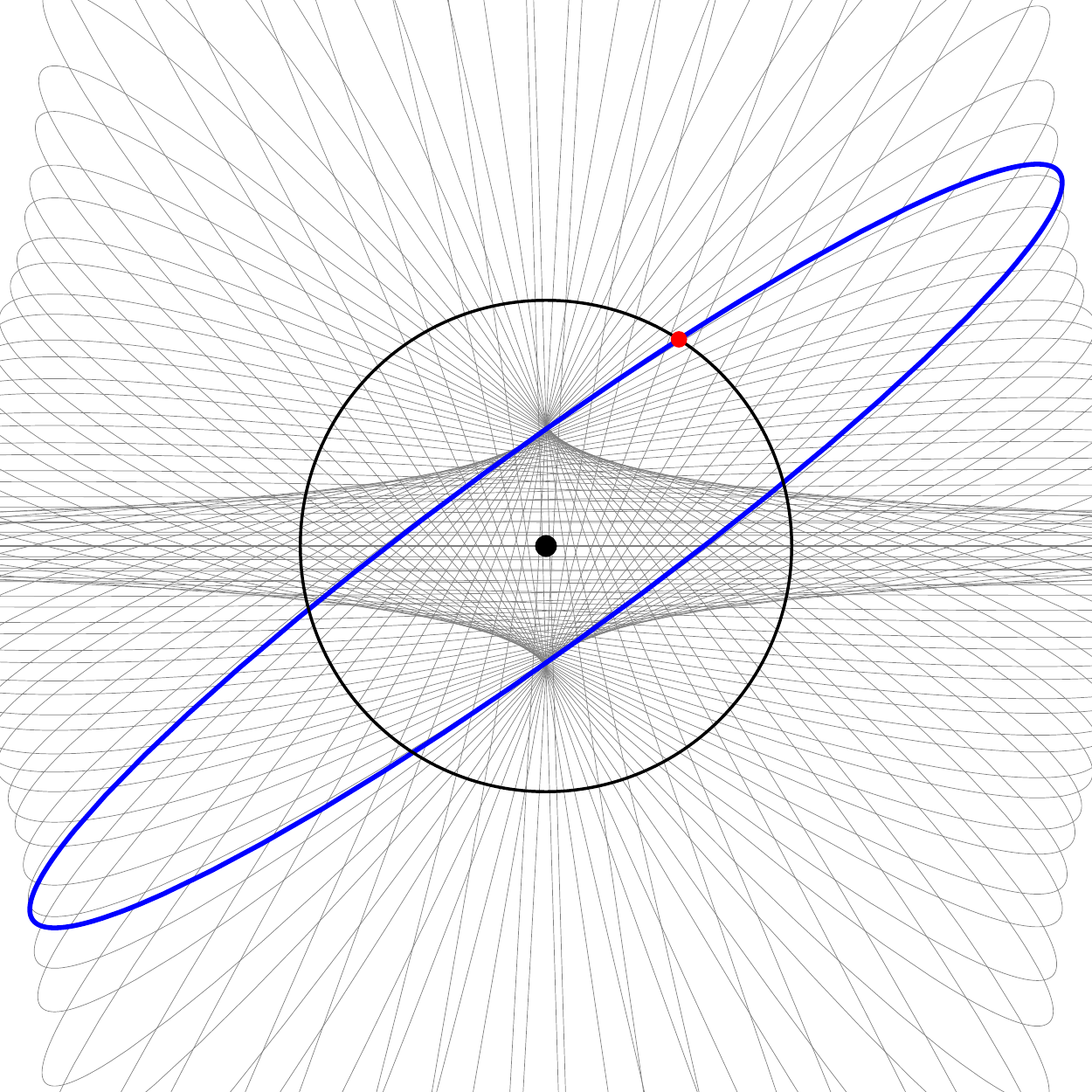}
\caption{Hooke's chains, given by Proposition \ref{prop:hc}, for $c=-1,1,2$. }\label{fig:hooke_chains}
\end{figure}

 \subsection{Horocycles in the hyperbolic plane}

The space of Hooke ellipses is  $\H=\{(E,F,G)\st EG-F^2=1, E>0\}$, the hyperboloid model of the hyperbolic plane. The curves in $\H$ of constant (hyperbolic) curvature 1 are called {\em horocycles} and are the  sections of $\H$  by planes parallel to a  generator of the cone   $EG-F^2=0$. In the upper half-plane model these are (Euclidean) circles tangent to the real axis.
\begin{lemma}
For each fixed $(x,y)\in\R^2\setminus \{0\}$, the set of Hooke ellipses passing through $(x,y)$ is a  horocycle in $\H$. This  defines a bijection between the punctured plane $\R^2\setminus \{0\}$ and the space of horocycles in $\H$.
\end{lemma}

\begin{proof}
For each  $(x,y)\in\R^2\setminus 0$, equation \eqref{eq:huh},
$$
 x^2 -2 a x y+ (a^2+b^2 )y^2=b,
$$ defines   in the upper half-plane $\{(a,b)\st b>0\}$  either the circle of radius ${1\over 2y^2}$ centered at $\left({x\over y}, {1\over 2y^2}\right)$ if $y\neq 0$, or the   horizonal line $b=x^2$ if $y=0.$ These are precisely all the horocycles of the   upper half plane model   of the hyperbolic plane.
\end{proof}

It follows that  the horocycle path geometry in $\H$ is dual to the path geometry in $\R^2\setminus \{0\}$ of Hooke ellipses of fixed area. Thus we can use the analysis of the previous section to determine the projection of the chains to $\H$.

\begin{proposition}\label{prop:horc}
Each chain of the horocycle path geometry, up to the action of $\SLt$, projects to a curve in the hyperbolic plane,  given in the upper half-plane model $\{(x,y)\st y>0\}$  by
%
\begin{multline}\label{eq:horoc}
(x^2+y^2)^2 - [4 c x +(c^2+4) y] (x^2+y^2)
+ (6 c^2-2) x^2+2 c^3 x y + 6 y^2 \\ - 4 c (c^2-1) x - (c^4-3 c^2+4) y + (c^2-1)^2 = 0
\end{multline}
where $c\neq 0.$ See Figure \ref{fig:hc2}.%
This curve is the    projection  of a chain in $\SLt$, the  solution to  equations \eqref{eq:sl2chains1} that passes through   $\id\in \SLt$. The projection of this chain to  the Hooke plane is the Hooke ellipse $(x-cy)^2+y^2=1$.   The   horocycles  along this chain, in the upper half plane model,   all pass through $(c,1)$, the point corresponding to  this Hooke ellipse.   The chains corresponding to $c$ and $-c$ are congruent via
 an outer automorphism of $\SLt$ (conjugation by $\diag(-1,1)\in\GL_2(\R)$), acting  by reflection about the $y$-axis.

\end{proposition}

\begin{figure}
\centering
\includegraphics[width=.3\textwidth]{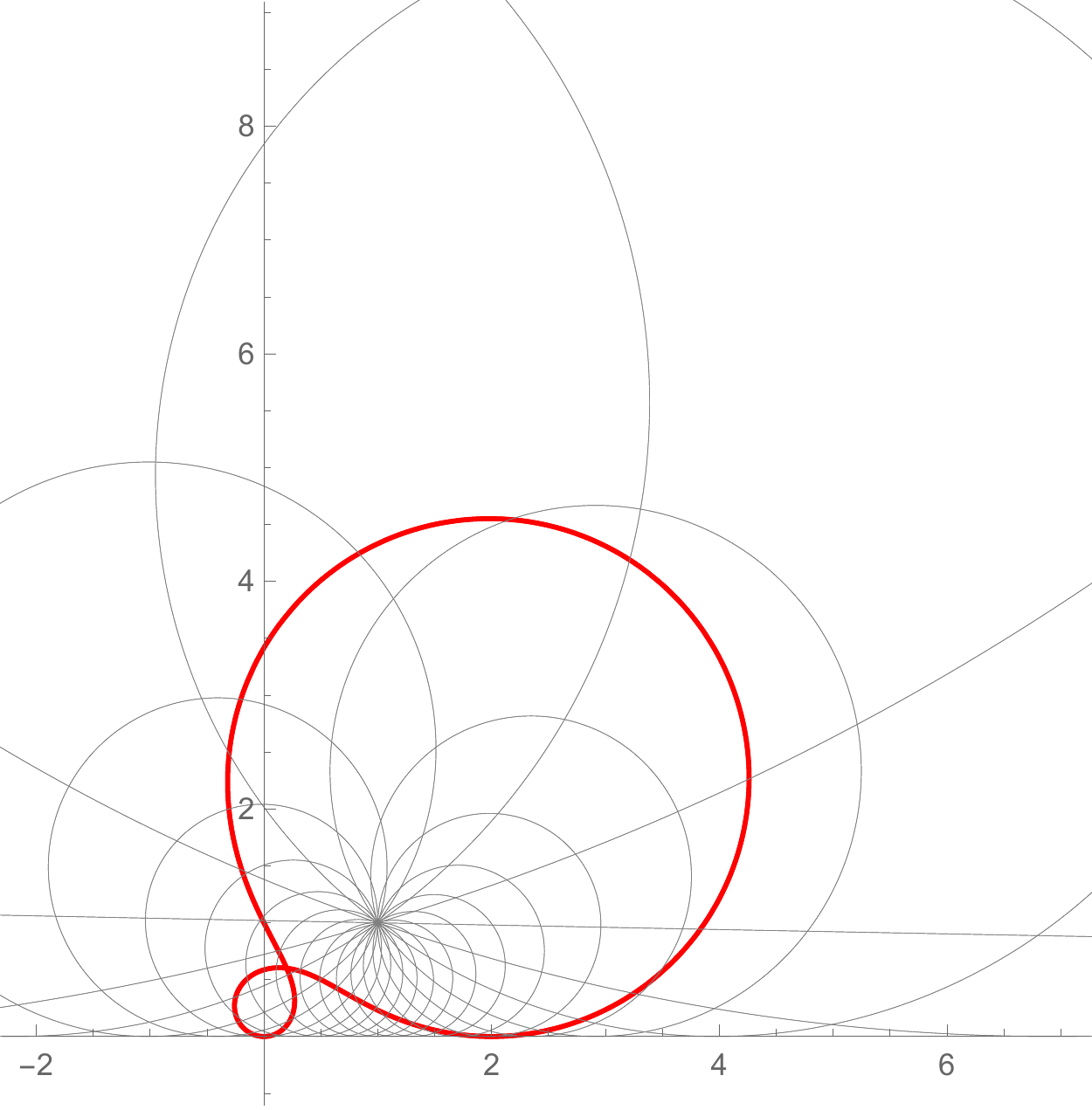}\quad
\includegraphics[width=.3\textwidth]{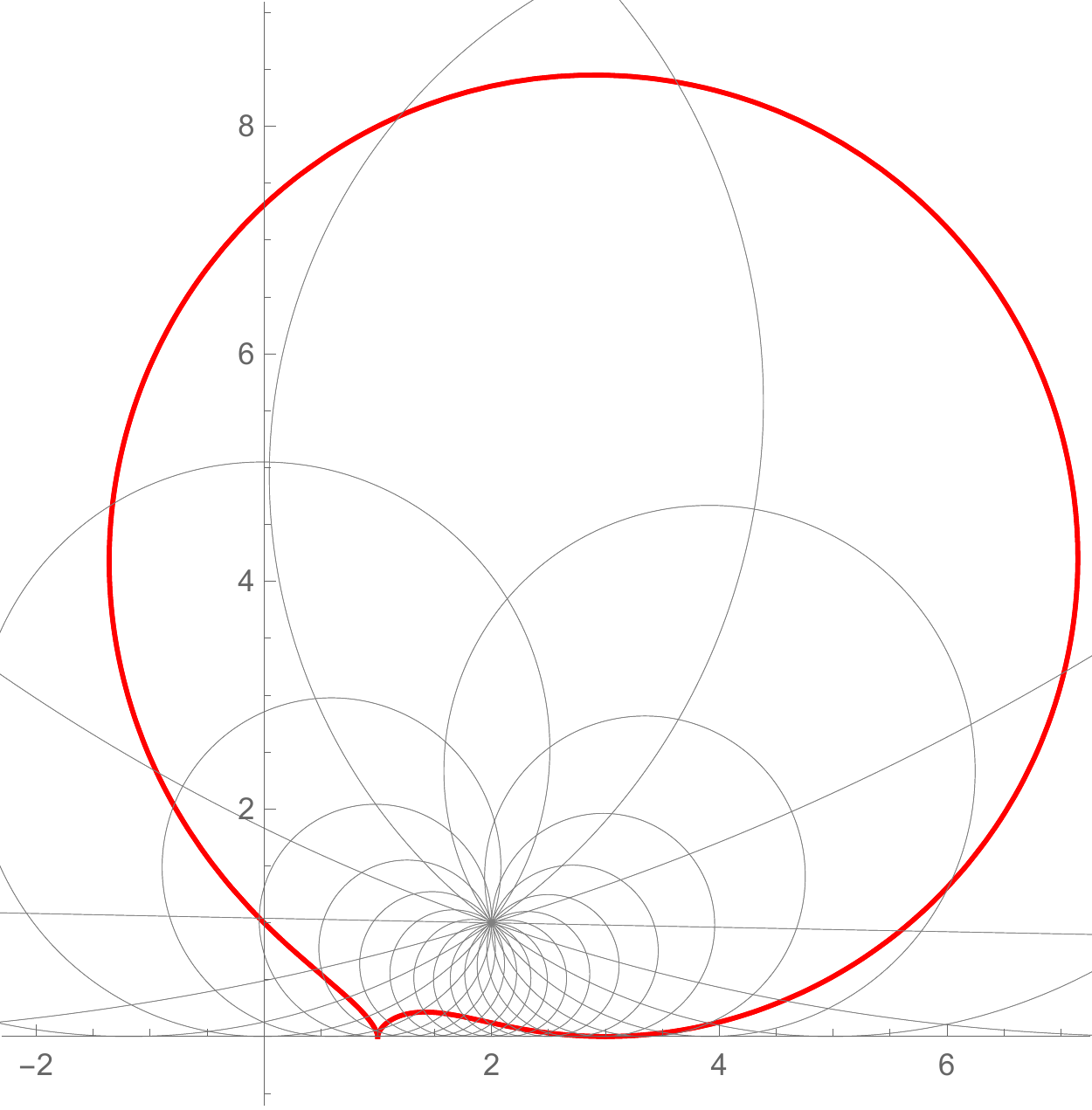}\quad
\includegraphics[width=.3\textwidth]{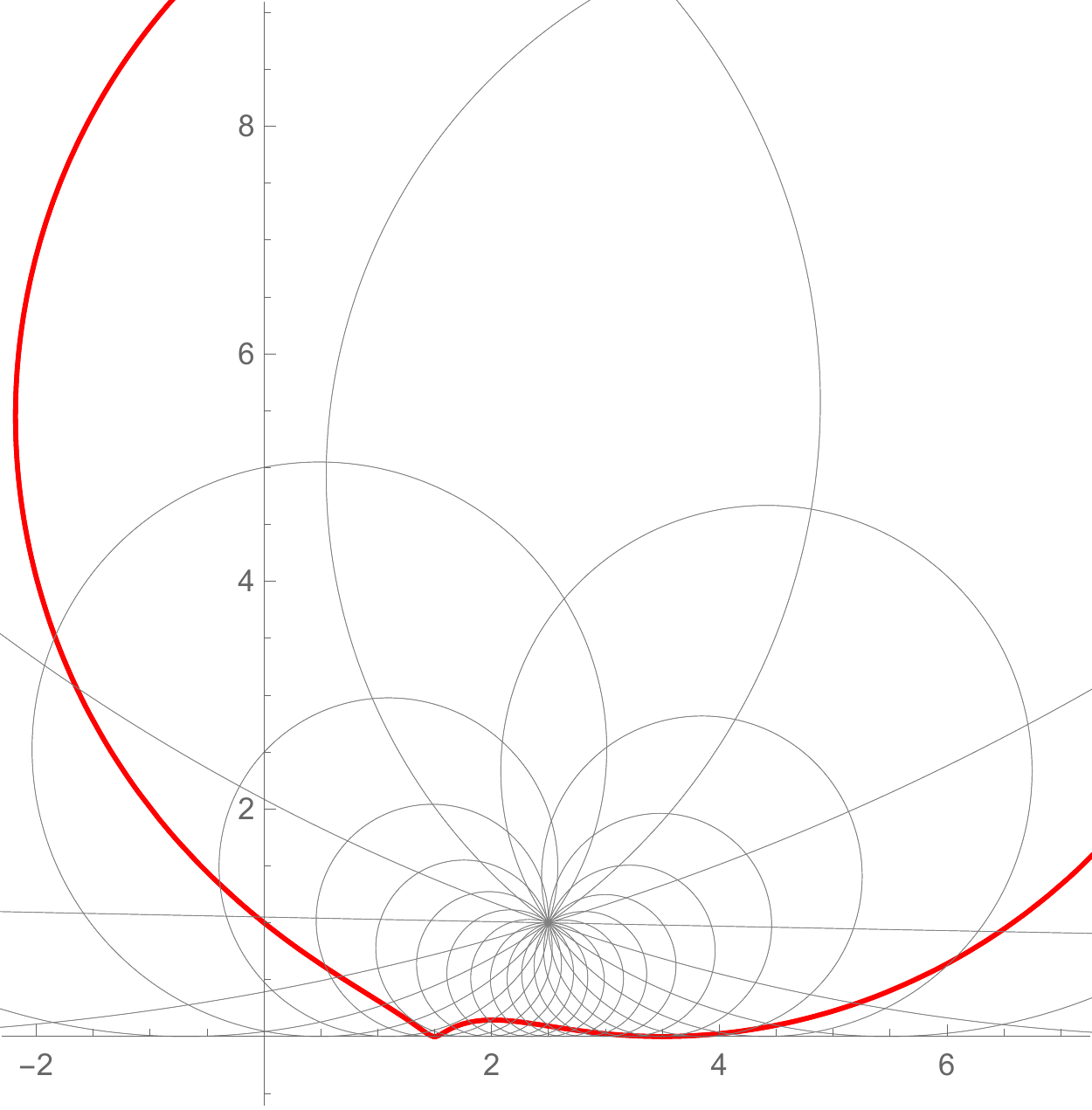}
\caption{Horocycle chains,  given by Proposition \ref{prop:horc}, projected to the hyperbolic plane (the upper half plane model), are rational bicircular quartics.  Left: {\em crunodal}  (one node), $|c|<2.$ Middle: {\em cuspidal} (one cusp), $c=2$. Right: {\em acnodal} (smooth), $|c|>2$.  }
\label{fig:hc2}
\end{figure}

\begin{proof}
Using $\r''=-\r$ and  $\r'=c(\sec2\tau)\r+\hb$ (Lemma \ref{lemma:proj} and equation \eqref{eq:sl2chains1}), the chain  $g(\tau)$   in $\SLt$ with $g(0)=\id$ is

\be\label{eq:gchain}
g(\tau)=\left(
\begin{array}{cc}
 \cos \tau+c \sin \tau & -\sin \tau \left(\sec (2 \tau ) c^2+\tan (2 \tau ) c+1\right) \\
 \sin \tau & \cos \tau-c \sec (2 \tau ) \sin \tau \\
\end{array}
\right).
\ee
The  projection  of this chain to  $\H$ is obtained by acting by $g(\tau)$ on the point in $\H$ corresponding to  the Hooke ellipse $\cE_0=\{x^2+y^2=1\}.$
One can check that the parametrization of $\H$ by the upper half-plane in equation \eqref{eq:uhp} is $\SLt$-equivariant, so one can act instead by $g(\tau)$ via fractional linear transformations on  $(0,1)$, the point in the upper half-plane corresponding to $\cE_0$. Reverting to  $\phi=2\tau$, the outcome is
$$
(x,y)={\left(c^{2}\left[(c +2 \sin \phi  ) \cos \phi -c -\sin \phi\right],\ -2  \cos^{2}\phi\right)
\over
-2  \cos^2 \phi  +c (c +2 \sin \phi  ) \cos \phi -c^2. }
$$
Eliminating $\phi$ in the above equation (we used Maple for this), one obtains equation \eqref{eq:horoc}.
\end{proof}

\begin{remark}The curves of Proposition \ref{prop:horc} are examples of {\em bicircular quartics}, a remarkable class of plane curves introduced by J. Casey in 1871 \cite{Ca}. They have many equivalent geometric and algebraic definitions, the simplest being {\em the inversion of a conic} (with respect to a circle).

\end{remark}

\end{document}